\documentclass{compositio}

\input{Dynkin.sty}

\usepackage{amsmath,amsfonts,amssymb,amsthm,amscd}
\usepackage{dsfont, mathrsfs,enumerate}
\usepackage{xspace}

\usepackage{eucal}

\usepackage[english]{babel}

\usepackage{xy, xypic}
\usepackage{color}
\usepackage{graphicx}
\usepackage{psfrag}

\usepackage{xcolor}
\definecolor{rouge}{rgb}{0.9,0.1,0}
\definecolor{bleu}{rgb}{0.1,0,0.9}
\definecolor{violet}{rgb}{0.7,0,0.8}
\usepackage[colorlinks=true,linkcolor=bleu,urlcolor=violet,citecolor=rouge]{hyperref}
\usepackage{breakurl}
\usepackage{xspace}

\newtheorem{thm}{Theorem}[section]
\newtheorem{lemma}[thm]{Lemma}

\newtheorem{prop}[thm]{Proposition}
\newtheorem{cor}[thm]{Corollary}
\theoremstyle{definition}
\newtheorem{defi}[thm]{Definition}
\newtheorem{rem}[thm]{Remark}
\newtheorem{question}[thm]{Question}
\newtheorem{ex}[thm]{Example}
\newtheorem{conj}[thm]{Conjecture}

\def\le{\leqslant}
\def\ge{\geqslant}

\def\F{\mathcal{F}}     

\def\L{\mathbb L}
\def\Z{\mathbb Z}
\def\Q{\mathbb Q}
\def\N{\mathbb N}
\def\R{\mathbb R}
\def\C{\mathbb C}

\def\A{\mathbb A}
\def\P{\mathbb P}
\def\O{\mathcal  O}      
\def\K{\mathcal K}      

\def\D{\mathcal{D}}    
\def\M{\mathcal{M}_{\C}}   
\def\E{\mathcal{E}}    
\def\V{{\rm Var}_{\mathbb C}}    

\DeclareMathOperator{\Spec}{Spec}

\def\ord{\mathrm{ord}}   

\def\root{\mathcal{R}}

\begin{document}

\title{The arc space of horospherical varieties and
motivic integration}

\author[Victor Batyrev]{Victor Batyrev}
\email{victor.batyrev@uni-tuebingen.de}
\address{Victor Batyrev, Mathematisches Institut, Universit\"at T\"ubingen,
 72076 T\"ubingen, Germany}

\author[Anne Moreau]{Anne Moreau}
\email{anne.moreau@math.univ-poitiers.fr}
\address{Anne Moreau, Laboratoire de Math{\'e}matiques et Applications,
 Universit{\'e} de Poitiers, France}

%

\begin{abstract}

For an arbitrary connected reductive group $G$ we consider the
motivic integral over the arc space of an arbitrary  $\Q$-Gorenstein
horospherical $G$-variety $X_\Sigma$ associated with
a colored fan $\Sigma$ and prove a formula for
the stringy $E$-function of $X_\Sigma$
which generalizes the one for toric varieties.
We remark that  in contrast to toric varieties the stringy $E$-function of
a Gorenstein horospherical variety $X_\Sigma$ may be not a polynomial
if some cones in $\Sigma$ have nonempty sets  of colors.
Using the stringy $E$-function, we can formulate and prove
a new smoothness criterion for locally factorial horospherical
varieties.
We expect that this smoothness criterion holds for
arbitrary spherical varieties.

\end{abstract}

\maketitle

\section*{Introduction}

Throughout the paper, we consider
algebraic varieties and algebraic groups
over the ground field $\C$.

\smallskip

Let $G$ be a connected reductive group and $H \subseteq G$ a closed subgroup.
The homogeneous space $G/H$ is called {\em horospherical} if $H$ contains
a maximal unipotent subgroup $U \subseteq G$.
In this case, the normalizer
$N_G(H)$ is a parabolic subgroup $P \subseteq G$ and $P/H$ is an algebraic
torus $T$.
The horospherical homogeneous space $G/H$ can be described
as a principal torus bundle with the fiber $T$ over the projective
homogeneous space  $G/P$.
The dimension $r$ of the torus $T$ is called
the {\it rank} of the horospherical homogeneous space $G/H$.
Let $M$ be the
lattice of characters of the torus $T$, and $N= {\rm Hom}(M, \Z)$ the
dual lattice.
According to the
Luna-Vust theory \cite{LV83}, any
$G$-equivariant embedding $G/H \hookrightarrow X$ of a
horospherical homogeneous
space $G/H$ can be described combinatorially by  a colored fan $\Sigma$ in
the $r$-dimensional vector space $N_\R:= N \otimes_\Z \R$.
In the case $H=U$,
$G$-equivariant embeddings of $G/U$ have been considered independently  by
Pauer \cite{Pau81,Pau83}.
Equivariant embeddings of horospherical homogeneous
spaces are generalizations of the well-known toric varieties which are
torus embeddings $T \hookrightarrow X$ ($G=T$, $H= \{e\}$).

Our paper is motivated by some known formulas for
stringy invariants of toric varieties.
Let $X$ be a $\Q$-Gorenstein toric variety defined by a fan
$\Sigma \subset N_\R$ and denote
by $|\Sigma| \subset N_\R$ its support.
Then there is a  piecewise linear function
$\omega_X\, : \, | \Sigma |  \to \R$ such that its restriction to every cone
$\sigma \in \Sigma$ is linear and $\omega_X$ has value $-1$ on all
primitive lattice generators of $1$-dimensional faces of $\sigma$.
It was shown in \cite{Ba98} that the
stringy $E$-function of the toric variety $X$ can be computed by the
formula
\begin{eqnarray} \label{Est-intro}
	E_{\rm st} (X ; u,v ) := \big( uv -1 \big)^r \sum_{n \in |\Sigma| \cap N} (uv)^{\omega_X(n)} .
\end{eqnarray}
If $X$ is smooth and projective, then the stringy $E$-function
of $X$ coincides with the usual $E$-function,
$$
	E(X ; u,v ) = \sum_{i = 1}^{ r} b_{2i} (X) (uv)^{i} \, ,
$$
where $b_{2i}(X)$ is the $2i$-th Betti number of $X$.
Using the decomposition of $X$ into torus orbits,
we can compute $E(X ; u,v)$ by the formula,
$$
	E (X ; u,v) = \sum\limits_{\sigma \in \Sigma} (uv - 1)^{r -\dim \sigma}
		 = \big(uv -1 \big)^r \sum\limits_{\sigma \in \Sigma}
		 	\displaystyle{\frac{(-1)^{\dim \sigma}}{(1 -uv)^{\dim \sigma}}} \, .
$$
Hence,
$$
	\sum_{n \in N} (uv)^{\omega_X(n)}
		= \sum\limits_{\sigma \in \Sigma}
		 	\displaystyle{\frac{(-1)^{\dim \sigma}}{(1 -uv)^{\dim \sigma}}}
		= (-1)^r P( R_\Sigma, uv)
		= (-1)^r \displaystyle{\frac{ \sum_{i = 1}^{ r} b_{2i} (X) (uv)^{i}}{(1 -uv)^r}} \, ,
$$
where $P (R_\Sigma, t) = \sum_{i \ge 0} \dim R_\Sigma^{i}\, t^{i}$
is the Poincar\'e series of the graded Stanley-Reisner ring
$R_\Sigma = \bigoplus_{i \ge 0} R_\Sigma^{i}$
associated with the fan $\Sigma$.

\smallskip

Recall the definition of the Stanley-Reisner ring
$R_\Sigma$.
Let $e_1,\ldots,e_s$ be the primitive integral generators
of all $1$-dimensional cones in $\Sigma$.
We consider the polynomial
ring $\C[z_1,\ldots,z_s]$ whose
the variables $z_1,\ldots,z_s$ are in
bijection to lattice vectors $e_1,\ldots,e_s$.
Then the {\em Stanley-Reisner ring} $R_\Sigma$ is
the quotient of $\C[z_1,\ldots,z_s]$
by the ideal generated by those square free monomials
$z_{i_1} \ldots z_{i_k}$
such that the lattice vectors
$e_{i_1} \ldots e_{i_k}$ do not generate any
$k$-dimensional cone in $\Sigma$.
The cohomology ring $H^\ast (X, \C)$ of
the smooth projective toric variety $X$ associated with $\Sigma$
is isomorphic to the quotient of $R_\Sigma$
modulo the ideal generated by a regular sequence $f_1, \ldots, f_r$
in $R_\Sigma^{1}$ (see e.g.\,\cite[Theorem 10.8]{D}).

\smallskip

In this paper, we prove a similar to (\ref{Est-intro}) formula
for any $\Q$-Gorenstein horospherical variety $X$ defined
by a colored fan $\Sigma$:
\begin{eqnarray} \label{Est-2}
	E_{\rm st} (X ; u,v ) := E(G/H ; u,v) \sum_{n \in |\Sigma| \cap N} (uv)^{\omega_X(n)} \, ,
\end{eqnarray}
where $\omega_X : |\Sigma| \to \R$ is a certain $\Sigma$-piecewise linear function
(cf.\,Theorem \ref{t:main}).
Let $X$ be a complete and locally factorial horospherical variety
defined by a colored cone $\Sigma$.
Let $e_1,\ldots,e_s$ be the primitive integral generators
of all $1$-dimensional cones in $\Sigma$.
Consider the positive integers $a_i := - \omega_X(e_i)$
for $i \in \{Ê1,\ldots, s\}$
and define the \emph{weighted Stanley-Reisner ring}
$R_\Sigma^w$
corresponding to the colored fan $\Sigma$
by putting $\deg z_i = a_i$
in the standard Stanley-Reisner ring $R_\Sigma$
(here we consider $\Sigma$ as an uncolored fan).
In Proposition \ref{p:SR}, we prove that
$$
	\sum_{n \in N} (uv)^{\omega_X(n)}
	= (-1)^r P(R_\Sigma^w, uv)
	= (-1)^r  \sum\limits_{\sigma \in \Sigma}
		\displaystyle{ \frac{(-1)^{\dim \sigma}}{ \prod_{e_i \in \sigma} \big(1 - (uv)^{a_i} \big)} } \, ,
$$
where $P(R_\Sigma^w, t)$ is the Poincar\'e series associated with
the weighted Stanley-Reisner ring $R_\Sigma^w$.
So we get
$$
	E_{\rm st} (X ; u,v ) = (-1)^r  E(G/H ; u,v) P(R_\Sigma^w, uv)  \, .
$$

In contrast to toric varieties, the stringy $E$-function of a locally factorial horospherical
variety $X$ needs not be a polynomial.
If $X$ is smooth, then $E_{\rm st}(X ; u,v) = E(X ; u,v)$
is polynomial and in particular the \emph{stringy
Euler number},  $e_{\rm st}(X) : = E_{\rm st} (X ; 1,1)$,
is equal to the usual Euler number $e(X) : = E(X ; 1,1)$.
If $X$ is a locally factorial horospherical variety
whose closed orbits are projective, then we show that
$e_{\rm st}(X) \ge e(X)$ and that the equality
 holds if and only if $X$ is smooth (cf.\,Theorem \ref{t:smo}).
We conjecture  that the equality
$$
	e_{\rm st}(X) =  e(X)
$$
can be used as a smoothness criterion for arbitrary
locally factorial spherical varieties (cf. Conjecture \ref{c:smo}).

\medskip

The key idea behind the formula (\ref{Est-2})
for toric varieties is the isomorphism
$$
	T({\mathcal K}) / T({\mathcal O}) \simeq N,
$$
where ${\mathcal O} := \C[[t]]$, ${\mathcal K} := \C((t))$ and
$T({\mathcal O})$ (resp.\,$T({\mathcal K})$) denotes the
set of ${\mathcal O}$-valued (resp.\,${\mathcal K}$-valued) points
in $T$.
We remark that the \emph{stringy motivic integral} over the arc space $X({\mathcal O})$
of a toric variety $X$ is equal to its restriction to the arc space
$T({\mathcal K})$.
The latter contains countably many
orbits of the maximal compact subgroup
$T({\mathcal O}) \subset T({\mathcal K})$ that are
parametrized by the elements $n$ of the lattice $N$.
The stringy motivic integral over a $T({\mathcal O})$-orbit corresponding
to an element $n \in N$ is equal to $(\L-1)^r \L^{\omega_X(n)}$
where $(\L-1)^r$ is the stringy motivic volume of the torus $T$
and $\L$ is the class of the affine line
in the Grothendieck ring K$_0({\rm Var}_\C)$ of algebraic varieties.
Our approach in the proof of the formula (\ref{Est-2})
is to use a more general bijection
$$
	G({\mathcal O}) \setminus (G/H)({\mathcal K}) \simeq N
 $$
which holds for any horospherical homogeneous space $G/H$,
see \cite{LV83} and \cite{GN10}.

\medskip

The paper is organized as follows.

\medskip

Section \ref{S:Mot}
contains a review of known facts about the spaces of arcs
of algebraic
varieties and their relation to motivic integrals
and stringy $E$-functions.
In Section \ref{S:Hor}, we collect basic results on horospherical
embeddings.
In Section \ref{S:Arc}, we prove that there is a bijection
between the quotient by $G(\O)$ of the intersection $X(\O) \cap  (G/H)(\K)$
and  the set of lattice points $|\Sigma| \cap  N$
for any horospherical $G/H$-embedding (cf.\,Theorem \ref{t:lattX}).
Section\,\ref{S:Est} is devoted to the formula
which expresses the stringy motivic volume
of any $\Q$-Gorenstein horospherical variety
as a sum over lattice points $n \in N \cap |\Sigma|$ (cf.\,Theorem \ref{t:main}).
We use this formula to obtain a smoothness criterion for locally factorial horospherical embeddings
in Section \ref{S:Smo}  (Theorem \ref{t:smo}).
Section \ref{S:App} contains some applications,
examples, open questions and a conjecture related to our results.

\medskip

\noindent
{\bf Acknowledgments}:
We would like to thank
M. Brion, B. Pasquier and D. Timashev for useful discussions and A. Szenes for his comments.
We are also indebted to the referee for his numerous and judicious comments
and his careful attention to our paper.

Our work was partially supported
by the DFG-project  "Geometrie und Kombinatorik von Toruswirkungen auf
algebraischen Variet\"{a}ten"
and by the ANR-project
10-BLAN-0110.

\section{Arc spaces, motivic integration and stringy motivic volumes}      \label{S:Mot}

Interesting invariants of a singular algebraic variety
$X$ can be obtained via
the nonarchimedean motivic integration over its space
of arcs  $\mathcal{J}_{\infty}(X)$.

\medskip

Here we recall the basic definitions on the arc space of an algebraic
variety and
refer the reader to \cite{DL99}, \cite{M01} or \cite{EM09} for
more details concerning  this topic.
Let $X$ be an algebraic variety  over $\C$.
For any $m \ge 0$, we denote by
$\mathcal{J}_m(X)$ the {\em $m\textrm{-}{\rm th}$ jet scheme} of $X$ over $\C$
whose $\C$-valued points are all morphisms of  schemes
$\Spec \C [ t ]/ (t^{m+1}) \to X.$
One has  $\mathcal{J}_0(X) = X$ and   $\mathcal{J}_1(X) = {\rm T}X$ is
the total space
of the tangent bundle over  $X$.
For $m \ge n$, the natural ring homomorphism
$\C[t]/ (t^{m+1})  \to \C[t]/(t ^{n+1})$
induces truncation morphisms
$$Ê\pi_{m,n} \, : \,  \mathcal{J}_m(X) \longrightarrow \mathcal{J}_n(X) . $$
The truncation morphisms
form a projective system
whose projective limit is an infinite dimensional
scheme $\mathcal{J}_\infty(X)$ over $\C$.
The scheme $\mathcal{J}_\infty(X)$ is called the {\em arc space} of $X$,
and the $\C$-valued points of $\mathcal{J}_\infty(X)$ are all morphisms
$\Spec \C [ [ t ]  ]  \to X .$
For each $m$, there is a natural  morphism
$$Ê\pi_m \, : \mathcal{J}_\infty(X) \longrightarrow \mathcal{J}_m(X) $$
induced by the ring homomorphism
$\C [ [ t ]  ]  \to \C[ [ t ] ]/ (t^{m+1}) \simeq
\C[ t ] /(t^{m+1}).$

\medskip

The motivic integration over the
arc space of a smooth variety is due to Kontsevich \cite{Ko95}.
One of its generalizations for singular varieties
was suggested by Denef and Loeser in \cite{DL99}.
Another generalization motivated by stringy invariants
was proposed in \cite{Ba98}; see also \cite{Cr} and~\cite{V}.

Let $\V$ be the category of complex algebraic varieties
and denote by K$_{0}(\V)$ the Grothendieck ring of $\V$.
For  an element $X$ in $\V$
we denote by $[X]$ its  class in K$_{0}(\V)$.
The symbol $\L$ stands for the class of the affine line $\A^1$
and we denote by $1$ the class of $\Spec \C$.
For example,
$$
	[Ê\mathbb{P}^n ] = \L^n + \L^{n-1} + \cdots + \L + 1  .
$$
The map $X \mapsto [X]$ naturally
extends to the category of constructible algebraic sets.
There is a natural function, $\dim : {\rm K}_0(\V) \to \Z \cup \{Ê\infty \}$,
which can be extended to the localization $\M := {\rm K}_{0}(\V)[\L^{-1}]$
of K$_{0}(\V)$ with respect to $\L$ simply
by setting $\dim(\L^{-1}) := -1$.
For any $m \in \Z$, set $F^m \M : = \{ \tau \in \M \ | \ \dim \tau \le m\}$.
Then $\{ F^m \M\}_{m \in \Z}$ is a decreasing filtration
of $\M$ and we denote by $\hat{\M}$ the separated completion of $\M$
with respect
to this filtration.

\medskip

Let  $X$ be a $d$-dimensional smooth variety.

\begin{defi}  \label{d:cyl}

A subset $C$ in $\mathcal{J}_\infty(X)$ is called {\it a cylinder} if
there are $m \in \N$
and a constructible subset $B_m \subseteq \mathcal{J}_m(X)$
such that $C = \pi_m^{-1}( B_m)$.
Such a set $B_m$ is called a {\em $m$-base} of $C$.

If $C \subseteq \mathcal{J}_\infty(X)$ is a cylinder
with $m$-base $B_m \subseteq \mathcal{J}_m(X)$,
we define its {\em motivic measure} $\mu_X(C)$ by
$$
	\mu _{X}(C) : = [ B_m ] \L^{- m d}
		= [\pi_m(C) ] \L^{- md} \ \in \, {\rm K}_0(\V).
$$

\end{defi}

This definition does not depends on $m$:
Indeed, because $X$ is smooth, the map
$$
	\pi_{n,m} \, : \,  \pi_{n}(C) \to \pi_m(C)
$$
is a locally trivial $\A^{(n-m)d}$-bundle for any $n \ge m$.
The collection
of cylinders forms an algebra of sets which means that $\mathcal{J}_\infty(X)$
is a cylinder and if $C,C'$ are cylinders, then also are
$\mathcal{J}_\infty(X) \smallsetminus  C$
and $C \cap C'$.
On the set on cylinders, the measure $\mu_X$ is additive on
finite disjoint unions.
Furthermore, for cylinders $C \subseteq C'$, one has
$\dim \mu_X(C) \le \dim \mu_X(C')$.

\begin{defi}

A subset $C \subset \mathcal{J}_\infty(X)$ is called {\em measurable}
if for all $n \in \N$ there is a cylinder $C_n$ and cylinders
$D_{n,i}$ for $i \in \N$ such that
$$
	C \bigtriangleup C_n \subseteq \bigcup_{i \in \N} D_{n,i}
$$
and $\dim \mu_X(D_{n,i}) \le - n$ for all $i$.
Here $C \bigtriangleup C_n = (C\smallsetminus C_n) \cup (C_n \smallsetminus C)$
denotes the symmetric difference of two sets.

If $C$ is measurable, we define its {\em motivic measure} $Ê\mu_X(C)$  by
$$
	\mu_X(C) :=	\lim_{n \to \infty} \mu_X(C_n) .
$$
This limit converges in $\hat{\M}$
and is independent of the $C_n$'s, cf.~\cite[Theorem 6.18]{Ba98}.

\end{defi}

\begin{prop} [{\cite[Prop.\,6.19 and 6.22]{Ba98}}]    \label{p:mot}

{\rm (i)} The measurable sets form an algebra of sets
and the motivic measure $\mu_X$ is additive on finite disjoint unions.
If $(C_i)_{i \in \N}$ is a disjoint sequence of measurable sets such that
$\lim_{i \to \infty} \mu_X(C_i) = 0$,
then $C: = \bigcup_{i\in\N} C_i$ is measurable and
$$\mu_X(C) = \sum_{i \in \N} \mu_X(C_i) .
$$

\smallskip

{\rm (ii)} Let $Y \subseteq X$ be a locally closed subvariety.
Then $\mathcal{J}_\infty(Y)$ is a measurable subset of $\mathcal{J}_\infty(X)$
and if $\dim Y < \dim X$ then $\mu_X(\mathcal{J}_\infty(Y))=0$.

\end{prop}

\begin{defi}

A function $F :  \mathcal{J}_\infty(X)  \to \Z \cup \{+\infty \}$
is called {\em measurable}
 if $F^{-1}(s)$ is measurable for all $s \in \Z \cup \{+\infty \}$.

\end{defi}
Let $A \subseteq \mathcal{J}_\infty(X)$ be a measurable set and
$F : \mathcal{J}_\infty(X) \to \Z \cup \{+\infty\}$ a measurable function
such that $\mu_X(F^{-1}(+\infty)) = 0$.
Then we set
$$	
	\int_{A} \L ^{-F} {\rm d}\mu_X :=
	\sum\limits_{s \in \Z} \mu_X( A \cap F^{-1}(s)) \L^{-s}
$$
in $\hat {\M}$ whenever the right hand side converges in $\hat{\M}$.
In this case, we say that $\L^{-F}$ is {\em integrable on} $A$.
To any subvariety $Y$ of $X$, one associates the {\em order function}
$$Ê
	\ord_Y \, : \,  \mathcal{J}_\infty(X) \to \N \cup \{\infty \}
$$
sending an arc $\nu \in \mathcal{J}_\infty(X)$ to the order of
vanishing of $\nu$ along $Y$.
An important example of an integrable function is the  function
$\L^{-\ord_Y}$ where $Y$ is
a  smooth hypersurface in $X$.

\medskip

We consider now the case where $X$ is a singular
normal irreducible variety.
Let $K_X$ be a canonical divisor of $X$.
Assume that $X$ is $\Q$-Gorenstein,
that is $m K_X$ is Cartier for some $m \in \N$.
Let $f :  X' \to X$ be a resolution of singularities of $X$
such that the exceptional locus of $f$ is a divisor 	
whose irreducible components $D_1,\ldots,D_l$ are smooth divisors
with only normal crossings,
and set
$$Ê
	K_{X'/X} := K_{X'} - f^*K_X = \sum\limits_{i=1}^l \nu_i D_i,
$$
where the rational numbers $\nu_i$ $(1 \leq i \leq l)$ are called
the discrepancies of $f$.
The rational numbers $\nu_i$  $(1 \leq i \leq l)$ can be computed  as follows.
Since $m K_X$ is Cartier,
we can consider $f^* (m K_X)$ as a pullback of the Cartier divisor
and write
$$
	m K_{X'} - f^* (m K_X) = \sum\limits_{i=1}^l n_i D_i
$$
with $n_i \in \Z$ for all $i$.
Then $K_{X'/X}$ can be viewed as an abbreviation of
the $\Q$-divisor $\sum\limits_{i=1}^l \nu_i D_i$
where $\nu_i := \frac{n_i}{m} $ for all $i$.
Assume further that $X$ has at worst log-terminal singularities,
that is $\nu_i > -1$ for all $i$ (cf.\,\cite{KMM}).
Set $I :=\{1,\ldots, l \}$ and for any subset $J \subseteq I$,
$$
	D_J : =
		\left\lbrace
		\begin{array}{ll}
			\bigcap_{j\in J} D_J & \textrm{if } J\not=\varnothing \\
			Y  & \textrm{if } J=\varnothing
		 \end{array}
		\right.
	\quad \textrm{ and  }\quad
	D_J^0 : = D_J \smallsetminus \bigcup \limits_{j \in I \smallsetminus J} D_j.
$$

\begin{defi}    \label{d:Est}

We define the {\em stringy motivic volume} $\E_{\rm st}(X)$
of $X$ by
$$
	 \E_{\rm st}(X) := \sum\limits_{J \subseteq \{1, \ldots , l \}}
[D_J^0] \prod\limits_{j \in J}
		\, \displaystyle{\frac{  \L - 1 }{ \L^{\nu_j+1} - 1}}
\, \in \, \hat{\M} (\L^{\frac{1}{m}}) .
$$
(In \cite{V}, the element $\E_{\rm st}(X)$ is also called the
{\em stringy $\E$-invariant} of $X$.)

\end{defi}

\medskip

The inequality $\nu_i > -1$ for any $i$ implies that
the function $\ord_{K_{X'/X}} : = \sum_{i=1}^l \nu_i \, \ord_{D_i}$
is integrable on $\mathcal{J}_\infty(X')$, see \cite[Theorem 6.28]{Ba98}.
So we can express $\E_{\rm st}(X)$ as a motivic integral:

\begin{prop}      \label{p:Est}

$$
	\E_{\rm st}(X) =
		\int_{\mathcal{J}_\infty(X')} \L^{-\ord_{K_{X'/X}}} \, {\rm d}\mu_{X'}
		\, \in \, \hat{\M} (\L^{\frac{1}{m}}) .
$$

\end{prop}

The crucial point is that the above expressions of $\E_{\rm st}(X)$
do not depend on the chosen resolution,
see \cite[Theorem 3.4]{Ba98}.
This relevant fact essentially comes from the transformation rule
for motivic integrals, see \cite{DL99}.

\medskip

Recall that the \emph{$E$-polynomial} of an arbitrary
$d$-dimensional complex
algebraic variety $Z$ is defined by
$$
	E(Z ; u,v) := \sum_{ p,q=0 }^d \sum_{i=0}^{2d}
		(-1)^{i} h^{p,q} ({\rm H}^{i}_{c}(Z ; {\C})) u^{p}v^{q} \, ,
$$
where $h^{p,q} ({\rm H}^{i}_{c}(Z ; {\C}))$ ($0 \le i \le 2d$)
is the dimension
of the $(p,q)$-type Hodge component in the $i$-th cohomology group
${\rm H}^{i}_{c}(Z ; {\C})$ with compact support.
The polynomial $E$ has properties similar to the ones of the usual Euler characteristic.
In particular, the map $Z \mapsto E(Z ; u,v)$ factors through the ring K$_{0}(\V)$.
The map $Z \mapsto E(Z ; u,v)$ extends to $\M$ by setting $E(\L^{-1} ; u,v) := (uv)^{-1}$.
So, we get a map from $\M$ to $\Z[u,v, (uv)^{-1}]$
which uniquely extends to  $\hat{\M}$. 
This extension will be again denoted by $E$.

\smallskip

\begin{defi}    \label{d:Est2}

The {\em stringy $E$-function} of $X$ is given by (cf.\,\cite{Ba98}):
$$
	E_{\rm st}(X ; u , v)
 		:=  \sum\limits_{J \subseteq \{1, \ldots , l \}} E(D_J^0;u,v) \prod\limits_{j \in J}
		\displaystyle{\frac{  uv-1 }{ (uv)^{\nu_j+1} -1 }}  \, .
$$

\end{defi}

\noindent
Note that  $E_{\rm st}(X;u,v) = E(\E_{\rm st}(X) ; u,v )$.

\medskip

\begin{rem}

Whenever $X$ is smooth, then $\E_{\rm st}(X)=\mu_X(\mathcal{J}_\infty(X))=[X]$
and $E_{\rm st}(X ; u,v ) = E(X ; u,v )$.

\end{rem}

\bigskip

\section{Horospherical varieties}          \label{S:Hor}

In this section, we use our notations from the introduction: $G$ is a connected reductive group over $\C$,
$H \subset G$ is a closed horospherical subgroup, $G/H$ is the corresponding
horospherical homogeneous space,   $U$ is a maximal unipotent subgroup in
$G$ such that $U \subseteq H$, $B := N_G(U)$ is the corresponding Borel subgroup of $G$,
$P:=N_G(H)$ is a parabolic subgroup, $T:=P/H$ a $r$-dimensional algebraic torus, $M$ is
the group of characters of $T$, and $N:= {\rm Hom}(M, \Z)$.
%
%
%
%

\smallskip

Let $S$ be the set of simple roots of $(G,B)$
with respect to a maximal torus of $B$.
There is a bijective map $I \mapsto P_I$
sending a subset $I$ of $S$ to the parabolic subgroup $P_I$ of $G$ containing $B$
such that $P_I = B W_I B$, where $W_I \subseteq W$ is the subgroup of the
Weyl group $W=W_S$ generated by the reflections $s_\alpha$ ($\alpha \in I)$. In particular,
 one has $P_\varnothing = B$ and $P_S=G$.
From now on, we denote by $I$ the subset of $S$ corresponding to $P:=N_G(H)$.
Let $U_0 \subset G/P$ be the open dense $B$-orbit.
Then $U_0$ is isomorphic to an affine space and the Picard
group of $G/P$ is free generated by
the classes $[\varGamma_\alpha]$
of irreducible
components $\{\varGamma_\alpha \ | \  \alpha \in S\smallsetminus I \}$
in the complement $(G/P) \smallsetminus U_0$. The space of global sections
$H^0( G/P, {\mathcal O}( \varGamma_\alpha))$ is an irreducible
representation of the universal cover of the semisimple group $G':= [G,G]$ corresponding to the fundamental weight $\varpi_\alpha$ associated with $\alpha \in S\smallsetminus I$.
Let $\phi \, : \, G/H \to G/P$ be the canonical surjective morphism whose fibers are isomorphic to
the torus $T$.
Then the divisors $\varDelta_\alpha := \phi^{-1}(\varGamma_\alpha)$,
for $\alpha \in S\smallsetminus I$,
are exactly the irreducible components in the complement to the
open dense $B$-orbit $\widetilde{U}_0 \simeq U_0 \times T$ in $G/H$.
The lattice $M$ can be identified with the group
$\C[\widetilde{U}_0]^* / \C^*$
of invertible
regular functions over $\widetilde{U}_0$ modulo nonzero constant functions.

\begin{defi}
A normal $G$-variety $X$ is said to be {\em horospherical}
if $G$ has an open orbit in $X$ isomorphic to the horospherical homogeneous space $G/H$.
In that case, $X$ is also called a {\em $G/H$-embedding}.
\end{defi}

Horospherical varieties are special examples of
spherical varieties.
%
According to the Luna-Vust theory \cite{LV83}, any
$G/H$-embedding $X$
can be described by  a colored fan $\Sigma$ in
the $r$-dimensional vector space $N_\R:= N \otimes_\Z \R$.
Our basic reference for spherical varieties is \cite{K91}.
For recent accounts about horospherical varieties,
see also \cite[Chap.\,1]{Pa07} or \cite[Chap.\,5]{Ti10}.

Let $X$ be a horospherical $G/H$-embedding.
Each irreducible divisor $D$ in $X$ defines a valuation $v_D : \C(X)^* \to \Z$
on the function field $\C(X)$ which vanishes on $\C^*$.
The restriction of $v_D$ to the lattice $M \simeq \C[\widetilde{U}_0]^* / \C^*$
yields an element $\varrho_D$ of the dual lattice $N$.

Let ${\mathcal X}(P)$ be the character group of the parabolic subgroup
$P=P_I$. This group
can be identified with the set of all characters $\chi \in {\mathcal X}(B)$ of the Borel
subgroup $B$
such that $\langle \chi, \check{\alpha} \rangle =0$
for all $\alpha \in I$ where $\check{\alpha} \in {\rm Hom}({\mathcal X}(B), \Z)$ denotes the coroot corresponding to $\alpha$.
Since every character of $P$ induces a line bundle
over $G/P$,  we get a homomorphism ${\mathcal X}(P) \to
{\rm Pic}(G/P)$.
Its composition with the monomorphism of character
groups $M \to {\mathcal X}(P)$,
induced by the epimorphism $P \to T = P/H$,
gives a homomorphism
$\delta \, : \, M  \to {\rm Pic}(G/P)$.
Let $\delta^* \, : \, {\rm Pic}(G/P)^* \to  N$ be the dual map.
Then, the lattice points
$\{\varrho_{\varDelta_\alpha} \ | \ \alpha \in S \smallsetminus I \} \subset N$
corresponding to the divisors $\varDelta_\alpha \subset X$,
$\alpha \in S \smallsetminus I$,
are exactly the $\delta^*$-images of the dual basis to $\{[\varGamma_\alpha] \ | \  \alpha \in S \smallsetminus I \}$
in $ {\rm Pic}(G/P)^*$.
For simplicity, we set $\varrho_\alpha := \varrho_{\varDelta_\alpha}$
for any $\alpha \in S \smallsetminus I$. We note that $\varrho_\alpha$ is equal to the restriction
to the sublattice $M \subseteq {\mathcal X}(B)$ of the corresponding
coroot $\check{\alpha}$.

Let  $\D_X = \{D_1, \ldots, D_t \}$ be the set of $G$-stable irreducible divisors of $X$.
For any divisor $D_i$, we denote by $\varrho_i$ the lattice point $\varrho_{D_i} \in N$.
Thus, we get a map
$$
	\varrho  \, : \,  \{\varDelta_\alpha \ | \ \alpha \inÊS \smallsetminus I \}
		\cup \D_X \to N
$$
which sends $\varDelta_\alpha$ ($\alpha \in S \smallsetminus I$)  to $\varrho_\alpha$
and $D_i \in \D_X$ ($1 \le i \le t$) to $\varrho_i$.
The restriction of $\varrho$ to $\D_X$ is injective,
but in general the restriction of $\varrho$ to $\{Ê\varDelta_\alpha \ | \ \alpha \inÊS \smallsetminus I \}$ is not injective.

\medskip

Let $Z$ be a $G$-orbit in $X$.
Denote by $X_Z$ the union of all $G$-orbits in $X$ which contain $Z$ in their closure.
Then $X_Z$ is open in $X$. Moreover, $X_Z$ is  a $G/H$-embedding having $Z$ as a unique
closed $G$-orbit. Such a  $G/H$-embedding is called {\em simple}.
It is well-known that any simple embedding is quasi-projective.
This fact follows from a result of Sumihiro \cite[Lemma 8]{Su74} which
states that
any normal $G$-variety is covered by $G$-invariant
quasi-projective open subsets  (if $X$ is a simple embedding
of $G/H$ with closed $G$-orbit $Y$, then any $G$-stable open neighborhood
of $Y$ in $X$ is the whole $X$).
The {\em colored cone corresponding to }$Z$ is the pair $(\sigma_{Z},\F_{Z})$ where
$\F_{Z}$ is the set $\{ \alpha \in S \smallsetminus I \ | \
\overline{\varDelta_\alpha} \supset Z\}$
and $\sigma_{Z}$ is the convex cone in  $N_\R$ generated by
$\{Ê\varrho_\alpha \ | \ \alpha \in \F_{Z} \}$ and
$\{ \varrho_i \ | \ D_i \supset Z \}$.
The {\em colored fan $\Sigma$ of $X$} is
the collection of the colored cones $(\sigma_{Z} , \F_{Z})$
where $Z$ runs through the set of $G$-orbits of $X$.
We call $\F : = \bigcup \, \F_{Z}$ the set of {\em colors of $X$}.

The set of colored cones in the colored fan $\Sigma$
is a partially ordered set:
We write $(\sigma',\F') \leq (\sigma,\F)$
and call $(\sigma',\F')$ a {\em face} of $(\sigma,\F)$
if $\sigma'$ is a face of $\sigma$
and $\F'   = \{ \alpha \in \F \ | \ \varrho_\alpha \in \sigma' \}$.
On the other hand, we have a partial order on the set of orbits,
$\big( Z \leq  Z' \iff Z \subseteq \overline{Z'} \big)$,
and the map $Z \mapsto (\sigma_{Z},\F_{Z})$ is an order-reversing bijection
between the set of orbits of $X$ and
the set of colored cones,~\cite{K91}.
Denote by $Z_{\sigma,\F}$ the $G$-orbit of $X$ corresponding
to $(\sigma,\F)$.
The open orbit $G/H$ corresponds to the cone $(0,\varnothing)$.

A arbitrary  pair $(\sigma,\F)$ consisting of a convex rational polyhedral
cone $\sigma \subset N_\R$
and a subset $\F \subset S \smallsetminus I$
is said to be a {\em strictly convex colored cone}
if $\sigma$ is strictly convex (i.e. $-\sigma \cap \sigma$ =0)
and if $\varrho_\alpha$ is a nonzero element in $\sigma$ for any $\alpha \in \F$.
A {\em colored fan} $\Sigma \subset N_\R$ is a collection of strictly convex
colored cones such that all faces of any colored cone $(\sigma, \F) \in
\Sigma$ belong to $\Sigma$
and the intersection of two colored cones is a common face of both cones, \cite[Section 3]{K91}.
The following result
was proved by Luna-Vust
in a more general context, \cite[Proposition 8.10]{LV83}
(see also \cite[Theorem 3.3]{K91}):

\begin{thm}   \label{t:bij}

The correspondence $X \to \Sigma$
is a bijection between $G$-equivariant isomorphism classes of $G/H$-embeddings
$X$ and colored fans $\Sigma$ in $N_\R$.

\end{thm}

We denote by $X_\Sigma$ the $G$-equivariant $G/H$-embedding
corresponding to a colored fan $\Sigma \subset N_{\R}$.
For simplicity, we denote $X_\Sigma$ by  $X_{\sigma,\F}$
whenever $\Sigma$ has only one maximal colored cone $(\sigma,\F)$.
The latter happens if and only if $X$ has a unique closed $G$-orbit,
i.e., $X$ is simple.

A horospherical $G/H$-embedding $X$
whose fan $\Sigma$ has no colors is said to be {\em toroidal}.
There is a simple method to construct a toroidal horospherical
variety associated with the (uncolored) fan $\Sigma$.
One considers the toric $T$-embedding
$Y_\Sigma$ with fan $\Sigma$.
Using the canonical epimorphism $P \to T$
we can consider $Y_\Sigma$ as a $P$-variety.
Then $X_\Sigma$ is isomorphic to the quotient
space $(G \times Y_\Sigma )/P$
where the action of $P$
on $G \times Y_\Sigma$
is given by
$p(g,y) := (g p^{-1}, py)$
for any $p \in P$, $g \in G$ and $y \in Y_\Sigma$.
One has a natural surjective morphism
$\phi : X_\Sigma \to G/P$
whose fibers are isomorphic to the toric variety $Y_\Sigma$
and $X \simeq X_\Sigma$.
Over the open dense $B$-orbit $U_0$ in $G/P$ the fibration
$\phi : \phi^{-1}(U_0) \to U_0$ is trivial.
Every toroidal horospherical variety is obtained as $(G_\Sigma \times Y)/P$
for a unique toric variety $Y_\Sigma$.
Moreover, $X_\Sigma$ is simple if and only if $Y_\Sigma$ is affine.

\smallskip

Each horospherical variety is dominated by a toroidal variety
in the following sense \cite[$\negmedspace$\S3.3]{Br91}:

\begin{prop}    \label{p:dom}

For any horospherical $G$-variety $X$,
there is a toroidal $G$-variety $\tilde{X}$ and a proper birational $G$-equivariant morphism
$$f \, :  \, \tilde{X}  \to X  .$$

\end{prop}
To obtain this toroidal variety $\tilde{X}$,
we just need to remove all colors from all colored cones in the fan of $X$.
It is worth mentioning that $\tilde{X} = (G \times Y )/P$, where $Y$ denotes the closure of $T$ in $X$.

In general, the toroidal variety $\tilde{X}$ is not smooth, but its singularities are locally isomorphic
to toric singularities.
In the sequel, it will useful to use a resolution of singularities
$f'  \, :  \, X'  \to X$, where $f'$
if a proper birational $G$-equivariant morphism
and where $X'$ is  a smooth toroidal $G$-equivariant embedding
with (uncolored) fan $\Sigma'$
obtained from $\Sigma$ by removing colors in all colored cones of $\Sigma$
and subdividing them into subcones generated by parts of $\Z$-bases of the lattice  $N$.
Note that the fans $\Sigma'$ and $\Sigma$ share the same support $|\Sigma|$.

\medskip

Next proposition describes the stabilizer of $G$-orbits
$Z_{\sigma,\F}$ in the horospherical case:

\begin{prop}    \label{p:orb}
Let $X$ be a horospherical $G/H$-embedding where $P:=N_G(H) = P_I$ is
the parabolic subgroup corresponding to a subset $I \subseteq S$.
Consider a colored cone  $(\sigma,\F) \in \Sigma$  $(\F \subseteq S \setminus I)$.
Define the sublattice $M_\sigma  : = M \cap \sigma^\perp$ consisted of
all elements in $M$ that are orthogonal to $\sigma \subset N_\R$.
Then every element $m \in M_\sigma$ defines a character $\chi_m$ of the
parabolic subgroup $P_{I \cup \F}$, and
the closed $G$-orbit $Z_{\sigma,\F} \subseteq X_{\sigma,\F}$
is isomorphic to $G/ H_{\sigma,\F}$
where
\[ H_{\sigma,\F} := \{ g \in P_{I \cup \F} \; | \;
\chi_m(g) =1 \;\; \forall m \in M_\sigma \}. \]
In particular, one has:
$$ \dim Z_{\sigma,\F} = {\rm rk} \, M_\sigma + \dim G/P_{I \cup \F}.$$
\end{prop}

\begin{proof}
First of all we recall that the nonzero elements
$\varrho_\alpha$ $(\alpha \in {\mathcal F})$
are the restrictions of the coroots $\check{\alpha}$ to the sublattice
$M \subseteq {\mathcal X}(B)$.
Since  $\varrho_\alpha \in \sigma$ for all $\alpha
\in {\mathcal F}$, the restriction of the coroot $ \check{\alpha}$ to $M_\sigma$
is zero for all $\alpha \in {\mathcal F}$.
The inclusions $M_\sigma \subseteq M \subseteq {\mathcal X}(P_I)$ imply
that the restriction of the coroot $\check{\alpha}$ to $M_\sigma$
is zero for all $\alpha \in I$, too.
Hence, we can consider the elements of $M_\sigma$ as characters of $B$ that
extend to the parabolic subgroup $P_{I \cup \F}$.

Without loss of generality, we can assume that $X = X_{\sigma,\F}$ is the simple
horospherical $G/H$-embedding corresponding to a colored
cone $(\sigma,\F)$.
Consider the  proper birational $G$-equivariant morphism
$f : X_{\sigma,\varnothing} \to X_{\sigma,\F}$
where $X_{\sigma,\varnothing}$ is the simple  toroidal variety
associated with the uncolored cone $({\sigma,\varnothing})$,
i.e.,  $X_{\sigma,\varnothing}$ is exactly the variety
$\widetilde{X_{\sigma,\F}}$ in the notations of Proposition \ref{p:dom}.

Then the toroidal simple horospherical variety $X_{\sigma,\varnothing}$
is a fibration
over $G/P$ with the affine toric fiber $Y_\sigma$. We remark
that  $f$ induces a bijection
between the set of $G$-orbits in $X_{\sigma,\varnothing}$ and
 the set of $G$-orbits in $X_{\sigma,\F}$.
It immediately follows
from the theory of toric varieties that the closed $T$-orbit
$Z_\sigma$ in $Y_\sigma$
is isomorphic to $T/T_{\sigma}$ where the subtorus $T_\sigma
\subseteq T$ is the kernel
of characters of $T$ in the sublattice $M_\sigma = M \cap \sigma^\perp$
of $M$.
Moreover, $Z_{\sigma, \varnothing}:= f^{-1}(Z_{\sigma, \F})$ is
 the closed $G$-orbit in $X_{\sigma,\varnothing}$ which is
isomorphic to $G \times_P (T/T_\sigma) $.  
%
%
This implies that
 the closed $G$-orbit
 $Z_{\sigma, \varnothing}$ is
isomorphic to $G/H_{\sigma,\varnothing}$ where
\[ H_{\sigma,\varnothing} := \{ g \in P=P_{I} \; | \;
\chi_m(g) =1 \;\; \forall m \in M_\sigma \}. \]

Let $z_0 \in Z_{\sigma, \varnothing}$ be a point with the stabilizer
$H_{\sigma,\varnothing}$. Then the stabilizer of $f(z_0) \in Z_{\sigma,\F}$
is a subgroup $H_{\sigma,\F} \subseteq G$ containing $H_{\sigma,\varnothing}$ so that we 
have the isomorphism $Z_{\sigma,\F} \cong
G/H_{\sigma,\F}$. 
We remark that all fibers of the proper birational $G$-equivariant morphism
$f : X_{\sigma,\varnothing} \to X_{\sigma,\F}$ are  connected
and proper.  In particular, $f$ induces a proper $G$-equivariant surjective
morphism of the $G$-orbits  $Z_{\sigma, \varnothing} \to  Z_{\sigma,\F}$
whose fibers are connected proper algebraic varieties isomorphic to
$H_{\sigma,\F}/H_{\sigma,\varnothing}$. 
Since the horospherical subgroup $H_{\sigma,\F}$ contains the horospherical subgroup $H_{\sigma,\varnothing}$, 
the normalizer $N_G( H_{\sigma,\F}) = : P_1$  contains the normalizer $N_G(H_{\sigma,\varnothing})=P$. 
%
%
Indeed, we have that $H_{\sigma,\varnothing} \supseteq [P,P]$ since $P/H_{\sigma,\varnothing}$ is commutative. 
It follows that $P_1 = B[P_1,P_1] = B H_{\sigma,\F} = P H_{\sigma,\F} \supseteq P$. 
Let $H'$ be the
intersection  $H_{\sigma,\F} \cap P$. The inclusions
\[H_{\sigma,\varnothing}  \subseteq H' \subseteq H_{\sigma,\F} \]
enable to decompose the proper morphism $f \, : \, G/H_{\sigma,\varnothing} \to
G/H_{\sigma,\F}$ into the composition of two proper morphisms with
connected fibers:
\[ f_1 \, : \, G/H_{\sigma,\varnothing} \to
G/H', \;\;    \;\; f_2 \, : \,G/H' \to  G/H_{\sigma,\F}. \]
The inclusions
$$[P,P] \subseteq H_{\sigma,\varnothing}  \subseteq H' \subset P$$
imply that the fibers of $f_1$ are isomorphic to a diagonalisable
subgroup $H'/H_{\sigma,\varnothing}$ in the torus $P/H_{\sigma,\varnothing}$.
But $H'/H_{\sigma,\varnothing}$ is connected and proper only if it consists
of one point, i.e., we get $H' :=H_{\sigma,\F} \cap P =
H_{\sigma,\varnothing}$. 
Let $M_1 \subset {\mathcal X}(P_1)$
be the sublattice of all characters of $P_1$  that vanish on
$H_{\sigma,\F}$. Since $P/[P,P]$ is a torus with the group 
of characters  ${\mathcal X}(P)$, it follows from the properties 
of diagonalisable groups that there exists one-to-one correspondence 
between the sublattices in  the group of characters ${\mathcal X}(P)$ 
and the closed subgroups in $P$ containing $[P,P]$. Therefore, 
the equality $H_{\sigma,\F} \cap P =
H_{\sigma,\varnothing}$ and the injectivity of the restriction map ${\mathcal X}(P_1) \to
{\mathcal X}(P)$  imply that   $M_1$ is also the
sublattice of all characters of $P$  that vanish on
$H_{\sigma,\varnothing}$, i.e., we get the equality $M_1 =M_\sigma$.
%

It remains to show that $P_1 = {P_{I \cup \F}}$. Since $P_1$ contains $P=P_I$, 
we get $P_1=P_J$ for some subset $J \subseteq S$ containing 
$I$. 
Let $\alpha \in S \smallsetminus I$. 
By the definition  of the set of colors $\F$,
the simple root $\alpha$ belongs to $\F$
if and only if the closure of the $B$-invariant divisor
$\Delta_\alpha:= \phi^{-1}(\varGamma_\alpha) \subset G/H$ 
in  $X_{\sigma, \F}$
contains the closed orbit $Z_{\sigma, \F} \subseteq X_{\sigma, \F}$.
On the other hand, the horospherical homogeneous $G$-space
$Z_{\sigma, \F}$ is a torus fibration over $G/P_J$,
and the intersection
of a closed $B$-invariant divisor
$\overline{\Delta_\alpha} \subset X_{\sigma, \F}$
with the closed $G$-orbit $Z_{\sigma, \F}$ is either a closed
$B$-invariant divisor  in  $Z_{\sigma, \F}$ (which projects to
a $B$-invariant divisor  in $G/P_J$), or the whole $G$-orbit
$Z_{\sigma, \F}$. 
The latter implies that $\overline{\Delta_\alpha}$ contains 
$Z_{\sigma, \F}$ (i.e., $\alpha \in \F)$ 
if and only if $\alpha \in J$. So we obtain  
$J = I  \cup \F$. 

%


\end{proof}

\bigskip

\section{Arcs spaces of horospherical varieties}     \label{S:Arc}

Let $\K := \C ( ( t ) )$ be the field of formal Laurent series,
and let $\O := \C [ [ t ] ]$ be the ring of formal power series.
If $X$ is a scheme of finite type over $\C$,
denote by $X(\K)$ and $X(\O)$ the sets of $\K$-valued points
and $\O$-valued points of $X$ respectively.
Remark that the set $X(\O)$ coincides
with the set of $\C$-points of the scheme $J_\infty(X)$.
If $X$ is a normal variety admitting an action of an algebraic group $A$,
then $X(\K)$ and $X(\O)$ both admit a canonical action of the group $A(\O)$
induced from the $A$-action on $X$.

\smallskip

The following result can be viewed as a generalization
in a slightly different context of \cite[\S8.2]{GN10}
(see also \cite{LV83} or \cite{Do09}):

\begin{thm}   \label{t:lattX}

Let $X$ be a horospherical $G/H$-embedding
defined by a colored fan $\Sigma$.
We consider the two sets $X(\O)$ and $(G/H)(\K)$
as subsets of $X(\K)$.
Then there is a surjective map
$$ \mathcal{V} \, : \, X(\O) \cap
(G/H)(\K)  \longrightarrow  |\Sigma| \cap  N
$$
whose fiber over any $n \in |\Sigma| \cap N$ is precisely one $G(\O)$-orbit.
In particular, we obtain
a one-to-one correspondence between the lattice points in $|\Sigma| \cap N$
and the $G(\O)$-orbits in $X(\O) \cap (G/H)(\K)$.

\end{thm}

In the special case where $X$ is a toric $T$-embedding,
Theorem \ref{t:lattX} is due to Ishii, \cite[Theorem 4.1]{Ish04}.
In more detail, by \cite[Theorem 4.1]{Ish04} (and its proof),
we have:

\begin{lemma}   \label{l:lattX}

Let $Y:=Y_\Sigma$ be a toric $T$-embedding
defined by a fan $\Sigma$.
For any $\K$-rational point $\lambda \in T(\K)$, we denote by  $\lambda^*$
the corresponding ring homomorphism
 $\lambda^* : \C[M] \to \K$ and
define
 the element $n_\lambda$ of the dual lattice
$N = {\rm Hom}(M,\Z)$ as the composition of
$\lambda^*|_M : M \to \K^*$ and the standard valuation map
$\ord :  \K^* \to \Z$.
Then the map
$$\nu\; : T(\K)  \rightarrow N,
\, \lambda \mapsto n_\lambda
$$
induces a canonical isomorphism $T(\K)/T(\O) \cong N$
and one obtains a surjective map
$$\nu\; : \; Y(\O) \cap  T(\K)  \rightarrow  |\Sigma| \cap  N,
\, \lambda \mapsto n_\lambda
$$
whose fiber over any $n \in |\Sigma| \cap N$ is precisely one $T(\O)$-orbit.

\end{lemma}

The above lemma will be used in the proof of Theorem \ref{t:lattX}:

\begin{proof}[Proof of Theorem \ref{t:lattX}]
%
Consider the canonial surjective morphism $\phi \, : \, G/H \to G/P$ whose
fibers are isomorphic to the algebraic torus $T:= P/H$. 
We consider $p_0:=[P]$
as a distinguished $\C$-point of $G/P$ such that  the fiber $\phi^{-1}(p_0) =T$
is the closed subvariety $P/H \subseteq G/H$.

Since $G/P$ is a projective variety, the valuative criterion of properness
implies that the natural map $(G/P)(\O) \to (G/P)(\K)$ from
$\O$-points of $G/P$
to $\K$-points of $G/P$ is bijective. 
It follows from  the local triviality of the map $G \to G/P$ 
that $(G/P)(\O) = G(\O)/P(\O)$. 
Thus, the group $G(\O)$ 
transitively acts on $G(\O)/P(\O) = (G/P)(\O) = (G/P)(\K)$.

Let $\lambda \in
(G/H)(\K)$ be a $\K$-point of $G/H$. Then 
$\phi(\lambda) \in (G/P)(\K) = G(\O)/P(\O)$. 
So there exists an element $\gamma \in
G(\O)$ such that $\gamma(\phi(\lambda)) = p_0 \in (G/P)(\C) \subset
(G/P)(\K)$. 
Since the morphism $\phi\, : \, G/H \to G/P$ commutes with the left $G$-action, 
the equality $\gamma(\phi(\lambda)) = p_0 = [P]$ 
implies that $\gamma(\lambda) \in T(\K) = (P/H)(\K) \subset (G/H)(\K)$. 

Now we   set $n_\lambda := \nu({\gamma(\lambda)})$
where $\nu$ is
the map $T(\K) \to N = {\rm Hom}(M,\Z) \cong T(\K)/T(\O)$
defined by Lemma \ref{l:lattX}.
 It is easy to see that  the lattice point $n_\lambda$ does no depend
on the choice of the element $\gamma \in G(\O)$.
Indeed,
if $\gamma' \in G(\O)$ is another element such that
$\gamma'(\phi(\lambda)) = p_0$ then the equality 
$\gamma'(\phi(\lambda)) = \gamma(\phi(\lambda)) =p_0$ implies 
that the element $\delta:= \gamma' \gamma^{-1}$ 
belongs to  $P(\O)$ and its 
image under the homomorphism  $P \to T = P/H$
is contained in $T(\O)$. So, we obtain that  the $\K$-points  
$\gamma'(\lambda),
 \gamma(\lambda) \in T(\K)$ define the same element  $n_\lambda \in
N=T(K)/T(\O)$.
Finally,
we get a map $\mathcal{V} : (G/H)(\K) \to N, \,
\lambda \mapsto n_\lambda$
which is constant on $G(\O)$-orbits.

Denote by $\widetilde{X}$ the toroidal embedding
of $G/H$ corresponding the decolorization $\widetilde{\Sigma}$
of $\Sigma$.
Let $f  :  \widetilde{X}  \to X $
be the proper birational $G$-equivariant morphism
as in Proposition \ref{p:dom}.
The valuative criterion of properness for $f$ implies
the equality
\[ \widetilde{X}(\O) \cap (G/H)(\K) = {X}(\O) \cap (G/H)(\K). \]
Since $|\Sigma| = | \widetilde{\Sigma}|$, it remains to prove the statement
only for the toroidal horospherical variety  $\widetilde{X}$.

Let $Y_{\widetilde{\Sigma}}$ be  the closure
of the torus $T = P/H \subset G/H$ in
$\widetilde{X}$. Recall that the toroidal  horospherical variety
$\widetilde{X}$ is a
homogeneous fiber bundle $G\times_P Y_{\widetilde{\Sigma}}$ over $G/P$
with fiber isomorphic to the toric variety $Y_{\widetilde{\Sigma}}$
(see the discussion after the Theorem \ref{t:bij} for that point).
This allows to consider  the set
 $Y_{\widetilde{\Sigma}}(\O) \cap T(\K)$ as a subset
of $\widetilde{X}(\O) \cap (G/H)(\K)$.
The restriction of  $\mathcal{V}$ to  $Y_{\widetilde{\Sigma}}(\O) \cap T(\K)$
is exactly the map $\nu\, : \,Y_{\widetilde{\Sigma}}(\O) \cap T(\K) \to
|\widetilde{\Sigma}| \cap N$ from Lemma \ref{l:lattX}.
So the image of $\mathcal{V}$ contains $|\widetilde{\Sigma}|$.

In the toric fibration  
$\phi\, : \, \widetilde{X} \to G/P$  the fiber 
$\phi^{-1}(p_0) \subset  \widetilde{X}$ is exactly the toric variety 
$Y_{\widetilde{\Sigma}}$ and the intersection $Y_{\widetilde{\Sigma}} \cap G/H$ 
is exactly  the torus $T = P/H$. 
Since the group $G(\O)$ acts transitively on
$(G/P)(\O) = (G/P)(\K)$, 
 for any $\lambda \in  \widetilde{X}(\O) \cap (G/H)(\K)$
there exists an element $\gamma \in G(\O)$ such that
$\gamma(\phi(\lambda))  = p_0$. This implies that 
$\gamma(\lambda) \in Y_{\widetilde{\Sigma}}(\O) \cap T(\K)$ and
$\mathcal{V}(\lambda) = \mathcal{V}(\gamma(\lambda))$. Therefore the images
of $\nu$ and  $\mathcal{V}$ are the same.

It  remains only to show that the fibers of $\mathcal{V}$
are precisely the $G(\O)$-orbits.
The latter follows from the $G(\O)$-action on $X(\O)$ and from
the canonical isomorphism
$G(\O) \setminus (G/H)(\K) \simeq N$
induced by $\mathcal{V}$,
see e.g.\,\cite[$\negmedspace$\S8.2]{GN10} (or \cite{LV83}),
because the subset $X(\O) \cap (G/H)(\K) \subset
(G/H)(\K)$ is $G(\O)$-invariant.
\end{proof}

We assume until the end of the section that $X$ is a smooth toroidal $G/H$-embedding such that
every closed orbit in $X$ is projective. This means that
$X$ corresponds to an uncolored fan $\Sigma$ such that every maximal cone of $\Sigma$ is generated
by a $\Z$-basis of $N$.
Then  $X$ is a fibration over $G/P$ with fiber
isomorphic to the smooth toric $T$-embedding $Y:=Y_\Sigma$
and the surjective map $\phi : X \to G/P$ induces,
for $m \in \N$,
surjective morphisms $\phi_m : \mathcal{J}_m(X) \to \mathcal{J}_m(G/P)$.
For any $m \in \N$, denote by
$\pi_m  :  \mathcal{J}_\infty(X) \to \mathcal{J}_m(X)$ and
$\pi'_m :  \mathcal{J}_\infty(Y) \to \mathcal{J}_m(Y)$
the canonical projection maps.
For any $n \in |\Sigma| \cap N$,
denote by $\mathcal{C}_{X,n}$ (resp.~$\mathcal{C}_{Y,n}$) the
$G(\O)$-orbit (resp.~$T(\O)$-orbit) of $X(\O) \cap (G/H)(\K)$ (resp.~$Y(\O) \cap T(\K)$)
corresponding to $n$ (see Theorem \ref{t:lattX} and Lemma \ref{l:lattX}).
As a consequence of the above proof of Theorem \ref{t:lattX}, we get:

\begin{cor}   \label{c:lattX}

Let $n \in |\Sigma| \cap N$ and $m \in \N$.
Then the restriction to $\pi_m(\mathcal{C}_{X,n})$
of $\phi_m$ is surjective onto $\mathcal{J}_m(G/P)$
and its fiber is isomorphic to $\pi'_m(\mathcal{C}_{Y,n})$.

\end{cor}

We aim to calculate the motivic measure (with respect to
$\mu_X$; cf. Definition \ref{d:cyl}) of the $G(\O)$-orbits in
$X(\O) \cap (G/H)(\K)$,
the other orbits having zero measure.

Let $n \in |\Sigma| \cap N$
and let $\sigma$ be a $r$-dimensional cone of $\Sigma$
such that $n \in \sigma$.
Fix a basis  $\{ u_1 ,\ldots, u_r \}$ of the semi-group $\sigma^\vee \cap M$.

\begin{lemma}   \label{l:meas}

Let $q \ge \max (\{  \langle n , u_j \rangle \ | \ j=1 ,\ldots, r \} )$.
In the notations of Corollary \ref{c:lattX},
the set $\mathcal{C}_{Y,n}$ is a cylinder with $q$-basis
$\pi_q' (\mathcal{C}_{Y,n}) \simeq (\A \smallsetminus 0 )^{r} \times \A^{q r - \sum\limits_{j=1}^r \langle n , u_j \rangle }.$

\end{lemma}

\begin{proof}

By our choice of $q$, for any $\nu \in \pi'_q(\mathcal{C}_{Y,n})$,
the truncated arc $\pi'_q(\nu)$
can be viewed as a $r$-tuple
$(\nu^{(1)},\ldots,\nu^{(r)})$
where
$$
	\nu^{(j)} = \nu^{(j)}_{\langle n, u_j  \rangle} t^{\langle n, u_j \rangle}
	+ \nu^{ (j) }_{\langle n, u_j  \rangle + 1} t^{\langle n, u_j  \rangle +1}
	+ \cdots + \nu^{(j)}_{q} t^{q} \,  ;  \quad j \in 1,\ldots,r  \, ,
$$
for $ \nu_{\langle n, u_j  \rangle}^{(j)} \in \C^*$ and
$(\nu^{(j)}_{\langle n, u_j  \rangle + 1}, \ldots, \nu^{(j)}_{q}) \in \C^{q- \langle n , u_j  \rangle}$.
Indeed, the orbit $\mathcal{C}_{Y,n}$ is the set of all arcs $\nu \in Y_\sigma(\O) \cap T(\K)$ such that $n_\nu = n$ (see Lemma \ref{l:lattX}).
So, the space of the truncated arcs $\pi'_q(\nu)$ is isomorphic to
$$
	(\A \smallsetminus 0 )^{r} \times \A^{\sum\limits_{j=1}^r (q- \langle n , u_j  \rangle) }
	=  (\A \smallsetminus 0 )^{r} \times \A^{qr - \sum\limits_{j=1}^r \langle n , u_j  \rangle } \, .
$$
Moreover, if $\nu \in Y(\O)$ lies in $\pi'^{-1}_{q}(\pi'_q (\mathcal{C}_{Y,n}))$
then $\nu \in \mathcal{C}_{Y,n}$.
Hence $\mathcal{C}_{Y,n}=\pi'^{-1}_{q}(\pi'_q (\mathcal{C}_{Y,n}))$
and $\mathcal{C}_{Y,n}$ is a cylinder whose $q$-basis
is the constructible set $\pi'_q(\mathcal{C}_{Y,n})$.

\end{proof}

\begin{thm}     \label{t:meas}

We have $\mu_{X} (\mathcal{C}_{X,n}) =  [ G/H ] \, \L^{ - \sum\limits_{j=1}^r \langle n , u_j \rangle }.$

\end{thm}

\begin{proof}

By Corollary \ref{c:lattX} and Definition \ref{d:cyl},
the motivic measure of the cylinder
$\mathcal{C}_{X,n}  = \pi_q^{-1}(\pi_q (\mathcal{C}_{X,n} ))$
of  $X(\O)$, for $q \gg 0$, is expressed by the formula:
$$\mu_{X}( \mathcal{C}_{X,n} ) = [ \pi_q( \mathcal{C}_{X,n} )] \L^{-q d}
 = [ \mathcal{J}_q(G/P)  ]  \, (\L-1)^r \, \L^{ q r - \sum\limits_{j=1}^r \langle n , u_j \rangle} \L^{-q d}  \, .$$
Since $\mathcal{J}_q(G/P)$ is a locally trivial
$\A^{q (d-r)}$-bundle over $G/P$
and $[ G/P ] (\L-1)^r  = [G/H ]$, we get
$\mu_{X} (\mathcal{C}_{X,n}) =  [ G/H ] \, \L^{ - \sum\limits_{j=1}^r \langle n , u_j \rangle }$.

\end{proof}

 \section{The stringy motivic volume of horospherical varieties}     \label{S:Est}

The aim of this section is to prove a formula for $\E_{\rm st}(X)$
for any $\Q$-Gorenstein
horospherical embedding $G/H \hookrightarrow X$,
see Theorem \ref{t:main}.

\smallskip

For our purpose, we need to explain the canonical class of a horospherical variety.
Let $G/H \hookrightarrow X$ be a $\Q$-Gorenstein $d$-dimensional
horospherical embedding.
For $\alpha \in S$, denote by $\varpi_\alpha$ the corresponding fundamental weight of $S$.
Let $\rho_S$ (resp.~$\rho_I$) be the half sum of positive roots of $S$ (resp.~$I$).
Note that $\rho_S =\sum_{\alpha \in S} \varpi_\alpha$.
For any $\alpha \in S\smallsetminus I$, we define the integers $a_\alpha$ by the equality:
$$
	2(\rho_S - \rho_I ) =\sum\limits_{\alpha \in S \smallsetminus I} a_\alpha \varpi_\alpha .
$$
We refer to  \cite[$\negmedspace$\S4.1]{Br93} or \cite[Theorem 4.2]{Br97} for the following result:

\begin{prop}              \label{p:KX}

Let $X$ be a $G/H$-embedding.
Then
$$
	K_X =  \sum_{\alpha \in S\smallsetminus I}  - a_\alpha \overline{\varDelta_\alpha} +
	\sum_{ j = 1}^t -D_j \, ,
$$
where $D_1, \ldots, D_t$ are the irreducible divisors in the complement of $X$
to the dense open $G$-orbit, and $\overline{\varDelta_\alpha}$
$(\alpha \in S\smallsetminus I)$ is the closure of $\varDelta_\alpha$ in $X$.

\end{prop}
Let $\Sigma \subset N_\R$ be the colored fan corresponding to $X$.
The $\Q$-Gorenstein property  is equivalent to  the existence of a continuous
function
$$Ê\omega_X \, : \, | \Sigma | \to \R $$
satisfying the following conditions (cf.~\cite[Proposition 4.1]{Br93}):

(P1) \; $\omega_X(e_\tau) = -1$ for a primitive integral generator $e_\tau$
of an uncolored ray $\tau$ of $\Sigma$;

(P2) \; $\omega_X(\varrho_\alpha) = - a_\alpha$ for a colored cone $(\sigma,\F)$ of $\Sigma$
and $\alpha \in \F$;

(P3) \; $\omega_X$ is linear on each cone $\sigma \in \Sigma$.

\medskip

Let $f' \, :  \, X'  \to X$ be a proper birational $G$-equivariant morphism
where $X'$ is  a smooth toroidal $G$-equivariant embedding
with (uncolored) fan $\Sigma'$
obtained from $\Sigma$ by removing colors and subdividing
(see the discussion after the proposition \ref{p:dom}).
Denote by
$$
	K_{X'/X} : = K_{X'}  - {f'}^* K_X
$$
the discrepancy divisor of $f'$.

\medskip

Let $\tau_1 ',\ldots,\tau_q '$ be the rays
of $\Sigma'$ which are not rays of $\Sigma$ (this set may be empty),
$e_{\tau_1'},\ldots, e_{\tau_q '}$ the respective primitive integral generators,
and $D'_1,\ldots, D'_q$ the respective irreducible $G$-stable
divisors of $X'$.
Let also $\tau_{1},\ldots,\tau_{t}$
be the uncolored rays of $\Sigma$
and $(\tau_{t+1}, \F_{t+1}),\ldots, (\tau_s ,\F_s)$ the colored ones.
Denote by $D_1,\ldots ,D_s$
the irreducible $G$-stable divisors of $X'$
corresponding to the rays $\tau_1,\ldots ,\tau_s$ of $\Sigma'$.
Thus,
$$
	\{ D'_1,\ldots, D'_m \} \cup \{ D_1, \ldots, D_s \}
$$
is the set of irreducible $G$-stable divisors of $X'$.
Let $e_{\tau_1}, \ldots, e_{\tau_s}$ be primitive integral generators of the rays
$\tau_1,\ldots,\tau_s$ of $\Sigma'$ respectively.

\begin{prop}  \label{p:div}

Assume that $X$ is $\Q$-Gorenstein.
Then
$$
	K_{X'/X} =   \sum\limits_{i=1}^{q} (- 1 - \omega_X(e_{\tau_i'}) ) D'_i
		+ \sum\limits_{j = t+1}^s (- 1- \omega_X( e_{\tau_j}) ) D_j  .
$$
Moreover, $K_{X'/X}$ is a smooth simple normal crossings Cartier divisor
and $X$ has at worst log-terminal singularities.

\end{prop}

\begin{proof}

Since $X'$ is smooth,
there is a continuous function, $\omega_{X'} : |\Sigma'| \to \R$,
satisfying the following conditions:

(P1$^\prime$) \; $\omega_{X'}(e_{\tau_i'}) = \omega_{X'}(e_{\tau_j}) = -1$
for all $i=1,\ldots, q$ and $j=1,\ldots,s$;

(P2$^\prime$) \; $\omega_{X'}$ is linear on each cone of $\Sigma'$.

\smallskip

\noindent
Define a function $\psi :  |\Sigma'| \to \R$
by setting $\psi  (n) : = \omega_{X'} (n) - \omega_X (n)$ for any $n \in N_\R$.
Then $\psi$ is a continuous map
which  is linear on each cone of $\Sigma'$ (use properties (P3) and (P2$^\prime$)).
By Proposition~\ref{p:KX},
$$
	K_{X'} =  \sum_{\alpha \in S\smallsetminus I}  - a_\alpha \overline{\varDelta_\alpha} +
	\sum_{ i = 1}^q -D_i' + \sum_{ j = 1}^s - D_j  \quad \textrm{ and }   \quad
    	K_{X} =  \sum_{\alpha \in S\smallsetminus I}  - a_\alpha \overline{\varDelta_\alpha} +
	\sum_{ j = 1}^t - D_j \, . $$
So, by the conditions (P1), (P2) and (P1$^\prime$), we get
$$
	K_{X'/X} = \sum\limits_{i=1}^{q} (- 1 - \omega_X(e_{\tau_i'}) ) D'_i
	+ \sum\limits_{j = t+1}^s (- 1 - \omega_X(e_{\tau_j}) ) D_j  \, .
$$
Since $X$ is $\Q$-Gorenstein, $X$ has at worst log-terminal singularities,
see \cite[Theorem 4.1]{Br93}.
At last, $X'$ being smooth and toroidal, $K_{X'/X}$ is a smooth
simple normal crossings divisor.

\end{proof}

We are now in the position to state the main result of this section:

\begin{thm}  \label{t:main}

Let $G/H \hookrightarrow X$ be a $\Q$-Gorenstein $d$-dimensional
horospherical embedding with colored fan $\Sigma \subset N_\R$,
and $\omega_X$ as above.
Then
$$
	\E_{\rm st}(X ) = [G/H ]
			\sum \limits_{ n \in | \Sigma| \cap N } \L^{\omega_X(n)} \, .
$$

\end{thm}

\medskip

The remaining of the section is devoted to the proof of Theorem \ref{t:main}:
Theorem \ref{t:main} will be a straightforward consequence of Lemma \ref{l:main1}
and Lemma \ref{l:main2}.
Keep the above notations and
denote by $\mathcal{C}_{X',n}$ the
$G(\O)$-orbit in $X'(\O) \cap (G/H)(\K)$ corresponding to $n \in |\Sigma| \cap N$
(cf.\,Theorem\,\ref{t:lattX}).

\begin{lemma}   \label{l:main1}

We have:
\begin{eqnarray*}
	\E_{\rm st}(X ) =
		\sum\limits_{n \in |\Sigma| \cap N}\  \int\limits_{ \ \mathcal{C}_{X',n}}
			\L^{- {\rm ord}_{K_{X'/X}}} \, {\rm d} \mu_{X'}.
\end{eqnarray*}

\end{lemma}

\begin{proof}

Since the $G(\O)$-orbits in $X'(\O)$
which are not contained in $(G/H)(\K)$
have zero motivic measure, we get by Definition \ref{d:Est}:
$$
	\E_{\rm st}(X) =  \int\limits_{X'(\O)} \L^{- {\rm ord}_{K_{X'/X}}} \, {\rm d} \mu_{X'}
		= \int\limits_{X'(\O) \cap (G/H)(\K)} \L^{- {\rm ord}_{K_{X'/X}}} \, {\rm d} \mu_{X'} \,  .
$$
In addition, by Theorem \ref{t:lattX}, $X'(\O) \cap (G/H)(\K)$ is a countable disjoint union of $G(\O)$-orbits
and each of these $G(\O)$-orbits corresponds to a point
$n \in |\Sigma| \cap N$:
$$
	X'(\O) \cap (G/H)(\K) =
		\bigsqcup \limits_{n \in |\Sigma| \cap N } \mathcal{C}_{X',n} .
$$
All $\mathcal{C}_{X',n}$ are cylinders 
whose union is a measurable set.
The lemma is then a consequence of Proposition \ref{p:mot}(i).

\end{proof}

\begin{lemma}   \label{l:main2}
For any lattice point $n \in |\Sigma| \cap N $,
we have
$$ \int\limits_{ \mathcal{C}_{X',n} } \L^{- {\rm ord}_{K_{X'/X}}} \, {\rm d} \mu_{X'} Ê
= [ G/H ]  \, \L^{\omega_X(n)}. $$

\end{lemma}

\begin{proof}

Let $(\sigma,\F)$ be a colored cone in $\Sigma$
such that $\sigma$ contains $n$.
We remark that the statement
of the lemma is local.
So, it is enough to prove it in the case where $X$ is
the simple horospherical variety corresponding to $(\sigma,\F)$.
Furthermore,
we can assume that $\sigma$ has the maximal dimension $r$ (i.e., the unique closed
$G$-orbit in $X$ is projective).
Otherwise we can embed $\sigma$ as a face into
some  $r$-dimensional cone $\hat{\sigma}$ such that the restriction of the
linear function $\omega_{\hat{X}}$ to $\sigma$ coincides with $\omega_X$ and the
smooth subdivision of $\sigma$ extends to a smooth subdivision of $\hat{\sigma}$.
Here, $\hat{X}$ is the simple horospherical $G/H$-embedding
corresponding to the $r$-dimensional colored cone
$(\widehat{\sigma},\F)$.
Thus, it is enough to consider the case where every maximal cone of $\Sigma'$
is generated by a $\Z$-basis of $N$.

For the sake of the simplicity, we set, in the notations of Proposition \ref{p:div}:
$c'_{i} : = -1 - \omega_X(e_{\tau_i'}) $, for $i \in \{1, \ldots, q\}$,
and $c_{j} : = - 1 - \omega_X (e_{\tau_j})$, for $j \in \{t +1,\ldots, s\}$.
Thus,
$$
	K_{X'/X} =  \sum\limits_{i=1}^{q} c'_{i} D'_i
			+ \sum\limits_{j = t+1}^s c_{j}  D_j \, .
$$
Let $n \in |Ê\Sigma| \cap N$.
By the definition of motivic integrals,
$$
	\int\limits_{ \mathcal{C}_{X',n} } \L^{- {\rm ord}_{K_{X'/X}}} \, {\rm d} \mu_{X'} Ê
		= \sum\limits_{\nu \in \mathbb{Q} } \mu_{X'} ( \{ \lambda \in \mathcal{C}_{X',n} \ | \
			{\rm ord}_{K_{X'/X}} (\lambda) = \nu \} )  \, \L^{ - \nu} \,  .
$$
Let $\sigma$ be a $r$-dimensional cone of $\Sigma'$ containing $n$
and generated by a basis $\{e_1,\ldots,e_{r}\}$ of $N$.

Its dual basis, $\{u_1,\ldots,u_r \}$,
is a basis of the semi-group $\sigma^\vee \cap M$.
Possibly renumbering the vectors $e_1,\ldots,e_{r}$,
we can assume that there exist $l \in \{ 1,\ldots, q\}$
and $k \in \{ 1 ,\ldots, s \}$ such that, in the notations of Proposition \ref{p:div},
$\{e_{1}, \ldots, e_{l}\}$ is a part of $\{e_{\tau'_1}, \ldots, e_{\tau'_q}\}$,
$\{e_{l+1}, \ldots, e_{l+k}\}$ is a part of $\{e_{\tau_1}, \ldots, e_{\tau_t}\}$
and $\{e_{l+k+1}, \ldots, e_{r}\}$ is a part of $\{e_{\tau_{t+1}}, \ldots, e_{\tau_s}\}$.

It follows from the description of $ \mathcal{C}_{X',n}$ (see the proof of Lemma\,\ref{l:meas})
that, for any $\lambda \in  \mathcal{C}_{X',n}$,
$$
	{\rm ord}_{K_{X'/X}} (\lambda) =  \sum\limits_{i=1}^{l} c'_{i} \langle n , u_i \rangle
						 + \sum\limits_{j=l+k+1}^{r} c_{j} \langle n , u_j \rangle .
$$
As a result, we get:
$$
	\int\limits_{ \mathcal{C}_{X',n} } \L^{- {\rm ord}_{K_{X'/X}}} \, {\rm d} \mu_{X'} Ê
		= \mu_{X'} (\mathcal{C}_{X',n} )
			\, \L^{ -   \sum\limits_{i=1}^{l} c'_{i} \langle n , u_i \rangle
	 		- \sum\limits_{j=l+k+1}^{r} c_{j} \langle n , u_j \rangle}.
$$
In addition, by Theorem \ref{t:meas},
$$Ê
	\mu_{X'} (\mathcal{C}_{X',n} )
		= [ G/H ] \, \L^{ - \sum\limits_{ j=1 }^{r} \langle n , u_j \rangle} \, .
$$
So, it only remains to show that
$\omega_ X ( n )= - \sum\limits_{ j=1 }^{r} \langle n , u_j \rangle
   -   \sum\limits_{i=1}^{l} c'_{i} \langle n , u_i \rangle
   -   \sum\limits_{j=l+k+1}^{r} c_{j} \langle n , u_j \rangle $.
By the properties (P1), (P2) and (P3) of $\omega_X$, one has:
$$Ê
	\omega_X ( n )  = \omega_X ( \sum\limits_{ j=1 }^{r} \langle n , u_j  \rangle e_{j})
		= \sum\limits_{i=1}^{l} \langle n , u_i \rangle  \omega_X(e_{i})
		- \sum\limits_{j=l+1}^{l+k}  \langle n , u_j \rangle
  		+ \sum\limits_{j=l+k+1}^{r}   \langle n , u_j \rangle \omega_X(e_{j})
$$
$$Ê
	\qquad \qquad \qquad = - \sum\limits_{ j=1 }^{r} \langle n , u_j \rangle
   		-   \sum\limits_{i=1}^{l} c'_{i} \langle n , u_i \rangle
   		-   \sum\limits_{j=l+k+1}^{r} c_{j} \langle n , u_j \rangle \, .
$$
Then, the expected expression for $\omega_X(n)$ follows.

\end{proof}

As noticed before, Lemma \ref{l:main1} together with Lemma \ref{l:main2}
complete the proof of Theorem \ref{t:main}.

\medskip

\begin{ex}   \label{ex:Q}

Consider the case where $G=SL_3(\C)$,
$B$ is the Borel subgroup of $G$
consisted of upper triangular matrices of $G$, $S=\{\beta_1,\beta_2\}$
and $H=U$.
Then $G/H$ is a quasi-affine homogeneous horospherical variety
whose affine closure is the $5$-dimensional affine quadric
$$ Q= \{ (x_1,x_2,x_3,y_1,y_2,y_3) \in \A^6 \ | \ x_1 y_1 + x_2 y_2 + x_3 y_3= 0 \} \, ; $$
$Q$ is the affine cone over the Grassmannian $G(2,4)$.
Denote by $\check{\beta_1}$ and $\check{\beta_2}$ the coroots of $\beta_1$ and $\beta_2$ respectively.
The representation of $SL_3(\C)$ on $\A^6$ is
the sum of two fundamental $3$-dimensional
irreducible representations with the dominant weights  $\varpi_{\beta_1}$,
$\varpi_{\beta_2}$ and $Q$ has for maximal colored cone
$(\sigma, \{\beta_1, \beta_2\})$
where $\sigma$ is the cone of $N_\R$ generated by ${\check{\beta}_1}|_M$
and ${\check{\beta}_2}|_M$.
The quadric $Q$ admits four $G$-orbits:
$0$, two copies of $\A^3 \smallsetminus 0$, and the dense orbit $G/U$.
We have $[G/U] =  (\L^2-1)(\L^3 -1)$.
Using this decomposition into $G$-orbits of $Q$,
one gets $[Q] = \L^2(\L^3 + \L -1)$.
%
On the other hand, by Theorem \ref{t:main},
$$ \E_{\rm st}(Q) = [ G/U ]\,  \left( \sum\limits_{k\ge 0} \L^{-2k} \right)^{\!\!2}
	=  \displaystyle{ \frac{ (\L^2-1)(\L^3 -1) }{ (1 - \L^{-2} )^2}}
	= \displaystyle{ \frac{\L^4  (\L^2+\L+1)  }{\L+1}} \, .
$$
	
\smallskip

Let us show how this result can be obtained
using resolutions of singularities of $Q$.
We consider two different resolutions:
the blowing-up of the point $0 \in Q$
and a decolorization of $Q$.

\smallskip

1) Let $p \, : \, \hat{Q}  \to Q$ be the blowing-up of $0 \in Q$
and $D$ the exceptional divisor.
We have $K_{\hat{Q}} - p^* K_Q = 3 D$ and
$$
	[ \hat{Q} \smallsetminus D ] = [Q] - 1 =  \L^2(\L^3 + \L -1) - 1  .
$$
On the other hand, $D \simeq G(2,4)$ and $[ D ]$ can be readily computed using the Betti numbers.
Then by Definition \ref{d:Est}, we get:
$$
	\E_{\rm st}(Q)
 = [ \hat{Q} \smallsetminus  D ] +  [ D ] \left( \displaystyle{\frac{\L-1}{\L^{4}-1}} \right)
 = \displaystyle{\frac{ \L^4 (\L^2 + \L +1) }{\L +1} } \, .
$$

\smallskip

2) Let $Q'$ be the smooth toroidal variety
corresponding to the uncolored fan obtained from $\Sigma$
and $f' : Q' \to Q$ the corresponding proper birational
$G$-morphism.
Note that $Q'$ is the homogeneous vector bundle on $G/B$
associated with the representation of $B$ on $\A^2$ with weights
the fundamental weights $\varpi_{\beta_1}$, $\varpi_{\beta_2}$.
The exceptional locus of $f'$ has two irreducible components,
$D_1$ and $D_2$,
and $K_{Q'/Q} = D_1 + D_2$.
The set $Q' \smallsetminus (D_1 \cup D_2)$
 is isomorphic to the open orbit $G/U$
and $D_1 \smallsetminus (D_1 \cap D_2)$ is a locally trivial fibration over
$\A^3 \smallsetminus 0$ with fiber $\P^1$.
Moreover, $D_1\cap D_2$ is the unique closed $G$-orbit
which is here isomorphic to $G/B$.
Hence, by Definition \ref{d:Est},
\begin{eqnarray*}
\E_{\rm st} (Q) = [Q' \smallsetminus (D_1 \cup D_2) ]
        + 2\, \displaystyle{\frac{ [ D_1 \smallsetminus (D_1 \cap D_2) ] }{\L+1}}
        + \displaystyle{\frac{ [ D_1 \cap D_2 ] }{(\L+1)^2}}
 =   \displaystyle{\frac{ \L^4 (\L^2+\L+1)}{\L+1}}  \, .
\end{eqnarray*}

\end{ex}

\bigskip

\section{Smoothness criterion}           \label{S:Smo}

We obtain in this section (Theorem \ref{t:smo}) a smoothness criterion for
locally factorial horospherical embeddings
in term of their stringy Euler numbers
(cf.\,Definition \ref{d:Eul}).
Since the smoothness condition is a local condition,
we can restrict our study to the case of simple horospherical embeddings.

\smallskip

Recall that a normal variety is called {\em locally factorial}
if any Weil divisor is a Cartier divisor.
The following criterion for the locally factorial condition
can be readily extracted from \cite[Proposition 3.1]{Br89} and \cite[Proposition 4.2]{Br93}:

\begin{thm}   \label{t:lf}

Let $X$ be a simple horospherical $G/H$-embedding with maximal cone $(\sigma,\F)$.
Then, $X$ is locally factorial if and only if
the following two conditions are satisfied:

{\rm (L1)}  the restriction to $\{\varDelta_\alpha \ | \ \alpha \in \F\}$ of the map $\varrho$ is injective;

{\rm (L2)} \! $\sigma$ is generated by part of a basis of $N$ which contains
all $\varrho_\alpha$ for $\alpha \in \F$.

\end{thm}

Recall that the {\em usual Euler number} $e(V)$ of any complex algebraic variety $V$ is defined by
$$
	e(V) : = E (V ; 1,1) .
$$

\begin{defi}     \label{d:Eul}

Let $X$ be a $d$-dimensional normal $\Q$-Gorenstein variety.
Adopt the notations of Definition \ref{d:Est}
and define the {\em stringy Euler number} $e_{\rm st}(X)$ of $X$ by
$$ Ê
	e_{\rm st}(X) :=  \sum\limits_{J \subseteq \{1,\ldots, l \}} e(D_J^0) \, \prod\limits_{j \in J}
		\, \displaystyle{\frac{ 1 }{ \nu_j +1 }}.
$$

\end{defi}

The stringy $E$-function of $X$ was defined in Definition \ref{d:Est2}.
Note that $e_{\rm st}(X)$ is nothing but $ E_{\rm st} (X ; 1 ,1)$.
We refer to \cite{Ba98} of \cite{Ba99} for more details about the stringy Euler numbers.

\begin{thm}   \label{t:smo}

Let $X$ be a simple locally factorial horospherical $G/H$-embedding.
Assume that the maximal cone associated with $X$ has dimension $r$.
Then one has $e_{\rm st}(X) \ge e (X)$, and
the equality holds if and only if $X$ is smooth.

\end{thm}

Our assumption that the maximal cone associated with $X$ has dimension $r$
means that the closed orbit of $X$ is projective.
The proof of Theorem\,\ref{t:smo}
will be achieved at the end of the section.

\begin{ex}

The affine quadric $Q$ introduced in Example\,\ref{ex:Q}
yields an example of horospherical variety which is locally factorial
but not smooth,
$$
	e_{\rm st} (Q) = \frac{3}{2}
		> e(Q) = 1.
$$

\end{ex}

\begin{ex}    \label{ex:Grass}

Here we give an example of a singular horospherical variety $X$
for which the stringy $E$-function is polynomial.

Consider the case where $G=SL_4(\C)$, $B$ is the set of upper triangular matrices of $G$
and set $S = \{\beta_1,\beta_2,\beta_3\}$.
The representation of $G$ on
$\C^4 \oplus \wedge^2 \C^4 $
is the sum of two fundamental representations with
the dominant weights $\varpi_{\beta_1}$ and $\varpi_{\beta_2}$.
The stabilizer of $(e_1 , e_1 \wedge e_2) \in  \C^4 \oplus \wedge^2 \C^4$
in $G$ is the horospherical subgroup $H = P_{ \{ \beta_3 \} }  \cap (\ker \varpi_{\beta_1} \cap \ker \varpi_{\beta_2})$
where $(e_1,e_2,e_3,e_4)$ is the canonical basis of $\C^4$.
We have $\dim G/H =7$ and ${\rm rk} \, G/H = 2$.
Let $X \subset \wedge^2 \C^5 \simeq \C^4 \oplus \wedge^2 \C^4$ be the closure of the $G$-orbit
of $(e_1,e_1 \wedge e_2)$ in $\C^4 \oplus \wedge^2 \C^4$.
Then $X$ is  the affine cone over the Grassmannian
$G(2,5)$ and contains three more $G$-orbits:
$(\wedge^2 \C^4 \smallsetminus 0)$,
$(\C^4 \smallsetminus 0)$ and $0$.
From this, we get:
$[X] = \L^7 + \L^5 - \L^2$.
The maximal colored cone corresponding to $X$
is  $(\sigma,\{ \beta_1, \beta_2 \})$
where $\sigma$ is the cone of $N_{\R}$ generated by ${\check{\beta}_1}|_M$ and ${\check{\beta}_2}|_M$.
We have $a_{\beta_1} = 2$ and $a_{\beta_2} = 3$.
Hence, by Theorem\,\ref{t:main},
$$
	\E_{\rm st}(X) =
	\, \displaystyle{\frac{ (\L-1)^2  \, (\L+1) \, (\L^2+1) \, (\L^2 + \L +1) }
				{(1 - \L^{-2})(1 - \L^{-3})}}  =  \L^5 (\L^2 +1) \, .
$$
We have,
$e_{\rm st}(X) = 2 > e(X) = 1.$

\end{ex}

For $S' \subseteq S$, denote by $\Gamma_{S'}$ the Dynkin diagram corresponding to $S'$;
the vertices of $\Gamma_{S'}$ are the elements of $S'$.
In \cite[\S3.5]{Pau83}, Pauer gives a smoothness criterion
for any $G/H$-embedding in the case where $H=U$;
for the general case, see  \cite[Theorem 2.6]{Pa07} or \cite[Theorem 28.10]{Ti10}.
Recall here the criterion:

\begin{prop}   \label{p:PP}

Let $X$ be a simple locally factorial
horospherical $G/H$-embedding with maximal colored cone $(\sigma,\F)$
and let $I \subseteq S$ be such that $N_G(H) = P_I$.
Then, $X$ is smooth if and only if any connected component $\Gamma$
of $\Gamma_{I \cup \F}$ verifies one of the following conditions:

\smallskip

{\rm (C1)} \; $\Gamma$ is a Dynkin diagram of type ${\bf A}_\ell$, $\ell \ge 1$,
	and $\Gamma$ contains exactly one vertex in $\F$ which is extremal:

$
\begin{Dynkin}
	\Dbloc{\Dbullet\Deast}
	\Dbloc{\Dcirc\Dwest\Deast}
	\Dbloc{\Dcirc\Dwest\Deast}
	\Dbloc{\Ddots}
	\Dbloc{\Dcirc\Dwest\Deast}
	\Dbloc{\Dcirc\Dwest\Deast}
	\Dbloc{\Dcirc\Dwest}
\end{Dynkin}
$

{\rm (C2)} \; $\Gamma$ is a Dynkin diagram of type ${\bf C}_{\ell}$, $\ell \ge 3$,
	and $\Gamma$ contains exactly one vertex in $\F$ which is the simple extremal one:

$
\begin{Dynkin}
	\Dbloc{\Dbullet\Deast}
	\Dbloc{\Dcirc\Dwest\Deast}
	\Dbloc{\Dcirc\Dwest\Deast}
	\Dbloc{\Ddots}
	\Dbloc{\Dcirc\Dwest\Deast}
	\Dbloc{\Dcirc\Dwest\Ddoubleeast}
	\Dleftarrow
	\Dbloc{\Dcirc\Ddoublewest}
\end{Dynkin}
$

{\rm (C3)} \; $\Gamma$ is any Dynkin diagram whose vertices are all in $I$.

\end{prop}

\begin{ex}

1) The standard representation $(\C^{\ell +1},\varpi_1)$ of $G=SL_{\ell +1}(\C)$
is a smooth affine horospherical variety corresponding to the situation (C1).
Namely, the dense orbit $\C^{\ell+1} \smallsetminus 0$ of $\C^{\ell +1}$
is isomorphic to $G/H$ where $H$ is the kernel in the standard maximal parabolic $P$
whose Levi part contains the $\alpha_j$-root subgroups, for $j=2,\ldots,\ell$,
of the restriction to $P$ of $\varpi_1$.

2) The standard representation $(\C^{2\ell},\varpi_1)$ of $G=Sp_{2\ell}(\C)$
is a smooth affine horospherical variety corresponding to the situation (C2).
We have the same description of the dense orbit as in 1):
The dense orbit $\C^{2\ell} \smallsetminus 0$ of $\C^{2\ell}$
is isomorphic to $G/H$ where $H$ is the kernel in the standard maximal parabolic $P$
whose Levi part contains the $\alpha_j$-root subgroups, for $j=2,\ldots,\ell$,
of the restriction to $P$ of $\varpi_1$.

3) The case where $\F$ is empty (situation (C3)) corresponds to locally factorial toroidal embeddings
which are known to be smooth.

\end{ex}

We state several technical lemmas useful for the proof of Theorem\,\ref{t:smo}.
Our main reference for basics on Lie algebras and root systems is \cite{OV}.
Assume that $\Gamma_S$ is connected.
Let $I$ be a subset of $S$ and
let us introduce standard related notations.

\textbullet \; We denote by $\root$ the root system of $G$,
by $\root^+$ the set of positive root of $\root$,
by $\root_{I}$
the root subsystem of $\root$ generated by $I$
and by $\root_{I}^+$ the set $\root_{I} \cap \root^+$.

\textbullet \;
For any $\gamma \in \root$, we denote by $\check{\gamma}$
its coroot, and set $\check{S}:=\{Ê\check{\beta} \, ; \, \beta \in S\}$.

\textbullet \; If $\Gamma_{I}$ is connected,
we denote by $W_{I}$ the Weyl group associated with
$\root_{I}$, that is the subgroup of $GL(V)$ where $V:=\Z \root_{I} \otimes_{\Z} \R$
	generated by the reflections,
	$$s_{\alpha}\, : \,  V  \rightarrow V , \
			x  \mapsto  x - \langle x,\check{\alpha} \rangle \, \alpha,
		\qquad \alpha \in I.$$

\textbullet \; The exponents of $S$ (or $\check{S}$) will be denoted by $m_1, \ldots , m_{\ell}$.
We can assume that $m_1 \le \cdots \le m_\ell$.
The integers $m_1 +1 ,\ldots, m_\ell +1$ are the degrees of the basic $W_S$-invariant polynomials and we have
	$$| W_{S} |= \prod_{i=1}^\ell (m_i+1). $$
In addition, $\sum_{i=1}^{\ell} m_i = |\root^+|$.

\textbullet \; For $\gamma \in \root^+$, the {\em height}
of $\gamma$ is ${\rm ht} ( \gamma) := \sum_{\beta \in S}  \langle \check{\varpi}_{\beta}, {\gamma} \rangle$
where for $\beta \in S$, $\check{\varpi}_{\beta}$
is the fundamental weight of $\check{S}$ corresponding to $\check{\beta}$.
We denote by $\theta_{S}$ the highest root of $S$
and by $\theta_{\check{S}}$ the highest root of $\check{S}$.
One has $m_\ell = {\rm ht}(\theta_{S}) =  {\rm ht}(\theta_{\check{S}})$.

\textbullet \;  We denote by $\rho_{I} := \frac{1}{2} \sum_{\gamma \in \root_{I}^+}  \gamma$
the half sum of positive roots of $I$.
We have $\rho_{S}  =  \sum_{\beta \in S} \varpi_\beta$
and $\langle \rho_I ,\check{\beta} \rangle =1$
for any $\beta \in I$.

\textbullet \; Set $J := S \smallsetminus I$.
	The integers $a_\alpha$, for $\alpha \in J$, are defined by:
	$$
		a_\alpha :=	2 \, \langle \rho_S -\rho_I , \check{\alpha} \rangle
			  =	2 -  2 \langle \rho_I , \check{\alpha} \rangle
			  =     2 - \sum_{\gamma \in \root_I^+}  \langle \gamma , \check{\alpha} \rangle \, .
	$$

A dominant weight $\mu$ 
is called {\em minuscule} if $\langle \mu, \theta_{\check{S}} \rangle =1$.
If $\mu$ is minuscule then there is $\beta \in S$ such
that $\mu = \varpi_{\beta}$, cf. \cite[Chapter VI, \S2, exercise 24]{Bo}.

\begin{lemma}    \label{l:root}

Let $\alpha \in J = S \smallsetminus I$.
Then, $a_\alpha \in \{2,\ldots, m_\ell + 1\}$.
Furthermore, the equality $a_\alpha = m_\ell +1$ holds
if and only if $J=\{\alpha\}$
and $\varpi_\alpha$ is minuscule,
that is if $\alpha$ is one of the simple roots
as described below:


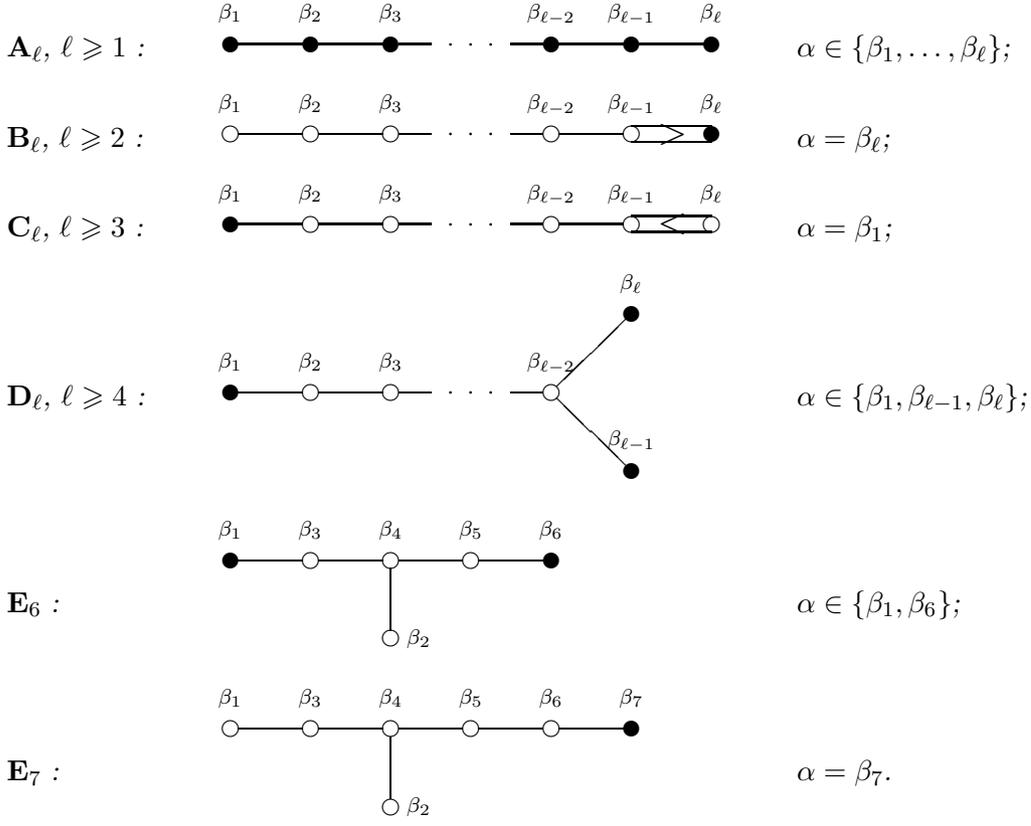
\begin{tabular}{llll}
&&&\\
${\bf A}_\ell$, $\ell \ge 1$  :&
$
\begin{Dynkin}
	\Dbloc{\Dbullet\Deast\Dtext{t}{\beta_1}}
	\Dbloc{\Dbullet\Dwest\Deast\Dtext{t}{\beta_2}}
	\Dbloc{\Dbullet\Dwest\Deast\Dtext{t}{\beta_3}}
	\Dbloc{\Ddots}
	\Dbloc{\Dbullet\Dwest\Deast\Dtext{t}{\beta_{\ell-2}}}
	\Dbloc{\Dbullet\Dwest\Deast\Dtext{t}{\beta_{\ell-1}}}
	\Dbloc{\Dbullet\Dwest\Dtext{t}{\beta_\ell}}
\end{Dynkin}
$
&  $\alpha \in \{ \beta_1,\ldots,\beta_\ell\}$;\\ 
${\bf B}_\ell$, $\ell \ge 2$ :&
$
\begin{Dynkin}
	\Dbloc{\Dcirc\Deast\Dtext{t}{\beta_1}}
	\Dbloc{\Dcirc\Dwest\Deast\Dtext{t}{\beta_2}}
	\Dbloc{\Dcirc\Dwest\Deast\Dtext{t}{\beta_3}}
	\Dbloc{\Ddots}
	\Dbloc{\Dcirc\Dwest\Deast\Dtext{t}{\beta_{\ell-2}}}
	\Dbloc{\Dcirc\Dwest\Ddoubleeast\Dtext{t}{\beta_{\ell-1}}}
	\Drightarrow
	\Dbloc{\Dbullet\Ddoublewest\Dtext{t}{\beta_\ell}}
\end{Dynkin}
$
& $\alpha=\beta_\ell$; \\ 
${\bf C}_\ell$,  $\ell \ge 3$ :&
$
\begin{Dynkin}
	\Dbloc{\Dbullet\Deast\Dtext{t}{\beta_1}}
	\Dbloc{\Dcirc\Dwest\Deast\Dtext{t}{\beta_2}}
	\Dbloc{\Dcirc\Dwest\Deast\Dtext{t}{\beta_3}}
	\Dbloc{\Ddots}
	\Dbloc{\Dcirc\Dwest\Deast\Dtext{t}{\beta_{\ell-2}}}
	\Dbloc{\Dcirc\Dwest\Ddoubleeast\Dtext{t}{\beta_{\ell-1}}}
	\Dleftarrow
	\Dbloc{\Dcirc\Ddoublewest\Dtext{t}{\beta_\ell}}
\end{Dynkin}
$  & $\alpha=\beta_1$; \\  
${\bf D}_\ell$, $\ell \ge 4$ : &
$
\begin{Dynkin}
	\Dspace\Dspace\Dspace\Dspace\Dspace
		\Dbloc{\Dbullet\Dsouthwest\Dtext{t}{\beta_{\ell}}}
	\Dskip
	\Dbloc{\Dbullet\Deast\Dtext{t}{\beta_1}}
	\Dbloc{\Dcirc\Dwest\Deast\Dtext{t}{\beta_2}}
	\Dbloc{\Dcirc\Dwest\Deast\Dtext{t}{\beta_3}}
	\Dbloc{\Ddots}
	\Dbloc{\Dcirc\Dwest\Dnortheast\Dsoutheast\Dtext{t}{\beta_{\ell-2}}}
	\Dskip
	\Dspace\Dspace\Dspace\Dspace\Dspace
		\Dbloc{\Dbullet\Dnorthwest\Dtext{t}{\beta_{\ell-1}}}
	\end{Dynkin}
$
& $\alpha \in \{ \beta_1,\beta_{\ell-1},\beta_\ell \}$;\\ 
${\bf E}_6$  :&$
\begin{Dynkin}
	\Dbloc{\Dbullet\Deast\Dtext{t}{\beta_1}}
	\Dbloc{\Dcirc\Dwest\Deast\Dtext{t}{\beta_3}}
	\Dbloc{\Dcirc\Dwest\Deast\Dsouth\Dtext{t}{\beta_4}}
	\Dbloc{\Dcirc\Dwest\Deast\Dtext{t}{\beta_5}}
	\Dbloc{\Dbullet\Dwest\Dtext{t}{\beta_6}}
	\Dskip
	\Dspace\Dspace\Dbloc{\Dcirc\Dnorth\Dtext{r}{\beta_2}}
\end{Dynkin}
$
&
$\alpha \in \{\beta_1,\beta_6\}$; \\ 
${\bf E}_7$ :&
$
\begin{Dynkin}
	\Dbloc{\Dcirc\Deast\Dtext{t}{\beta_1}}
	\Dbloc{\Dcirc\Dwest\Deast\Dtext{t}{\beta_3}}
	\Dbloc{\Dcirc\Dwest\Deast\Dsouth\Dtext{t}{\beta_4}}
	\Dbloc{\Dcirc\Dwest\Deast\Dtext{t}{\beta_5}}
	\Dbloc{\Dcirc\Dwest\Deast\Dtext{t}{\beta_6}}
	\Dbloc{\Dbullet\Dwest\Dtext{t}{\beta_7}}
	\Dskip
	\Dspace\Dspace\Dbloc{\Dcirc\Dnorth\Dtext{r}{\beta_2}}
\end{Dynkin}$
	&  $\alpha=\beta_7$. \\ 
\end{tabular}

\end{lemma}

\begin{proof}

Let $\alpha \in J$.
To begin with, since the coefficients of the Cartan matrix of $S$
are nonpositive outside the diagonal,  one has
$a_\alpha \ge 2$.
Moreover,
$a_\alpha \le 2 -
 2 \langle \rho_{S \smallsetminus \{ \alpha \}} , \check{\alpha} \rangle$.
Hence, we may assume that $J = \{ \alpha \}$,
i.e., $I = S \smallsetminus \{Ê\alpha\}$.
Consider now the two cases
depending on whether $\varpi_\alpha$ is minuscule or not.

\textasteriskcentered \; Assume that $\varpi_\alpha$ is not minuscule, i.e., $ \langle \varpi_\alpha, \theta_{\check{S}} \rangle >1$.
Then we have
\begin{eqnarray*}
m_\ell + 1  =   {\rm ht}(\theta_{\check{S}}) +1
= \sum_{\beta \in S} \langle \varpi_\beta, \theta_{\check{S}} \rangle +1
& = & \langle \varpi_\alpha, \theta_{\check{S}} \rangle
+  \sum_{\beta \in I } \langle \varpi_\beta, \theta_{\check{S}} \rangle +1\\
& > & 2 + \sum_{\beta \in I } \langle \varpi_\beta, \theta_{\check{S}} \rangle
 =  2 +  \langle \rho_I , \theta_{\check{S}} - \langle \varpi_\alpha , \theta_{\check{S}} \rangle \check{\alpha} \rangle .
\end{eqnarray*}
Since $\varpi_\alpha$ is not minuscule, $\langle \varpi_\alpha , \theta_{\check{S}} \rangle \ge 2$.
So,
$\langle \rho_I , - \langle \varpi_\alpha , \theta_{\check{S}} \rangle \check{\alpha} \rangle
\ge -2 \langle \rho_I , \check{\alpha} \rangle$ because $- \langle \rho_I , \check{\alpha} \rangle \ge 0$.
On the other hand,  one has $\langle \rho_I , \theta_{\check{S}} \rangle \ge 0$.
Otherwise there would be $\beta \in I$ such that $\langle \beta , \theta_{\check{S}} \rangle < 0$
which is impossible since $\theta_{\check{S}}$ is the highest root.
In conclusion, we get
$
m_\ell + 1 > 2 - 2  \langle \rho_I , \check{\alpha} \rangle = a_\alpha$
as desired.

\textasteriskcentered \; Assume that $\varpi_\alpha$ is minuscule, i.e., $ \langle \varpi_\alpha, \theta_{\check{S}} \rangle = 1$.
Then,
$a_\alpha = 2 \langle \rho_S - \rho_I , \check{\alpha} \rangle
 =   2 \langle \rho_S - \rho_I , \check{\alpha} +  \sum_{\beta \in I } \langle \varpi_\beta, \theta_{\check{S}} \rangle \check{\beta} \rangle
 = 2 \langle \rho_S - \rho_I , \theta_{\check{S}} \rangle$.
 Hence, we have
\begin{eqnarray*}
a_\alpha
 =  2 \langle \rho_S - \rho_I , \theta_{\check{S}} \rangle = {\rm ht}(\theta_{\check{S}}) + \langle \rho_S  - \rho_I ,\theta_{\check{S}} \rangle - \langle   \rho_I , \theta_{\check{S}} \rangle
=  m_\ell + 1 + \frac{1}{2} (
\sum_{\gamma \in \root^+ \smallsetminus  \root_I \atop\check{\gamma} \not= \theta_{\check{S}} } \langle \gamma   , \theta_{\check{S}} \rangle
- \sum_{\delta \in \root_{I}^+ } \langle   \delta, \theta_{\check{S}} \rangle )
\end{eqnarray*}
since $\langle \gamma   , \theta_{\check{S}} \rangle = 2$ whenever $\check{\gamma} = \theta_{\check{S}}$.
Then, our goal is to show that
$$
 \sum_{\gamma \in \root^+ \smallsetminus  \root_I \atop \check{\gamma} \not= \theta_{\check{S}} } \langle \gamma  , \theta_{\check{S}} \rangle
= \sum_{\delta \in  \root_I^+ } \langle   \delta, \theta_{\check{S}} \rangle .
$$
For any $\gamma \in \root^+$, 
we have $\langle \gamma, \theta_{\check{S}} \rangle \ge 0$ since $\theta_{\check{S}}$
is the highest root.
Set $\root' := \{Ê\gamma \in \root_{S}^+ \smallsetminus  \root_I \; | \;  \gamma \not= \theta_{\check{S}}
\textrm{ and } \langle \gamma, \theta_{\check{S}} \rangle > 0 \}$
and  $\root'' := \{Ê\delta \in \root_I^+ \; | \;  \langle \delta, \theta_{\check{S}} \rangle > 0 \}$.
Then we have to show the equality:
\begin{eqnarray}  \label{eq:root}
 \sum_{\gamma \in \root' } \langle \gamma  , \theta_{\check{S}} \rangle
= \sum_{\delta \in  \root'' } \langle   \delta, \theta_{\check{S}} \rangle .
\end{eqnarray}

Let $\gamma \in \root'$.
Since $\langle \gamma, \theta_{\check{S}} \rangle > 0$,
$\check{\delta} = \theta_{\check{S}} - \check{\gamma} $ is a root
of $\check{S}$
and $\theta_{\check{S}} - \check{\delta}$ is a root too.
In particular, $\langle \delta, \theta_{\check{S}} \rangle > 0$.
Next, show that $\delta \in \root_I^+$.

Since $\gamma \not \in \root_I$, $\check{\gamma} \not \in \root_{\check{I}}$.
Moreover, since $\varpi_\alpha$
is minuscule, $\langle \varpi_\alpha, \check{\gamma} \rangle = \langle \varpi_\alpha, \theta_{\check{S}} \rangle =1$.
So, $\check{\delta} = \theta_{\check{S}} - \check{\gamma} \in \root_{\check{I}}^+$
and then $\delta \in \root_I^+$.
Conversely, if $\delta \in \root''$, then $\check{\gamma} = \theta_{\check{S}} -\check{\delta}$
is a root and so $\langle \gamma, \theta_{\check{S}} \rangle > 0$.
Moreover, $\gamma$ is clearly an element of $\root_{S}^+ \smallsetminus  \root_I^+$
which is different from $\theta_{\check{S}}$,
that is $\gamma \in \root'$.
Therefore, the map from $\root'$ to $\root''$
sending $\gamma$ to $\delta$, where $\check{\delta} = \theta_{\check{S}} -\check{\gamma}$,
gives a bijection between the sets $\root'$ and $\root''$.
So, in order to prove the equality (\ref{eq:root}),
it remains to show that for any $\gamma \in \root'$,
we have $\langle \gamma, \theta_{\check{S}} \rangle = \langle \delta, \theta_{\check{S}} \rangle$
where $\check{\delta} = \theta_{\check{S}} -\check{\gamma}$.

Let $\gamma \in \root'$ and set $p := \langle \gamma, \theta_{\check{S}} \rangle > 0$.
Then the $\check{\gamma}$-string through
$\theta_{\check{S}}$ is $\{ \theta_{\check{S}}, \ldots, \theta_{\check{S}} - p\check{\gamma} \}$.
Since there is no minuscule weight in type {\bf G}$_{2}$,
we have $p \in \{Ê1,2\}$.
If $p =1$, then $\theta_{\check{S}}$ and $\theta_{\check{S}} - \check{\gamma} = \check{\delta}$
are roots but not $\theta_{\check{S}} - 2\check{\gamma} = \check{\delta} - \check{\gamma} = - (\theta_{\check{S}} - 2\check{\delta})$.
So, the $\check{\delta}$-string through
$\theta_{\check{S}}$ is $\{Ê\theta_{\check{S}}, \theta_{\check{S}} - \check{\delta} \}$
and $\langle \delta, \theta_{\check{S}} \rangle = 1$.
If $p=2$, then $\theta_{\check{S}}$, $\theta_{\check{S}} - \check{\gamma} = \check{\delta}$
and $\theta_{\check{S}} - 2\check{\gamma} = \check{\delta} - \check{\gamma} = - (\theta_{\check{S}} - 2\check{\delta})$
are roots.
So $\langle \delta, \theta_{\check{S}} \rangle \ge 2$ and then $\langle \delta, \theta_{\check{S}} \rangle = 2$.
Hence, in both cases, we have obtained
that $\langle \delta, \theta_{\check{S}} \rangle = p = \langle \gamma, \theta_{\check{S}} \rangle$
and the equality (\ref{eq:root}) is proven.

In conclusion, if $\varpi_\alpha$ is minuscule,
we have showed that $a_\alpha = m_\ell +1$.

\end{proof}

\begin{lemma}   \label{l:ht}

Let $S'$ be a subset of $S$
such that $\Gamma_{S'}$ is connected
and denote by $m'_1  \le \cdots \le m'_{l}$ the exponents of $S'$.
Then we have $m'_j \le m_j$ for any $j \in \{1,\ldots,l\}$.
In particular, ${\rm ht}( \theta_{S'} )  \le m_l$.

\end{lemma}

\begin{proof}

By a classical result, \cite{Ko59}, the partition of $|\mathcal{R}^+|$ formed by the exponents is
the dual to that formed by the number of positive roots of each height.
This easily implies the statement.

\end{proof}

Let $k$ be the cardinality of $I$,
and $m_1', \ldots , m_k ' $ the union
of all the exponents of subsets $S'$ such
that $\Gamma_{S'}$ is a connected component of $\Gamma_I$.
Order them so that
$m_1' \le \cdots \le m_k '.$
Number the roots $\alpha_{k+1},\ldots,\alpha_\ell$ of $J$
so that
$a_{\alpha_{k+1}} \le \cdots \le a_{\alpha_\ell}$
and set for simplicity $a_j : = a_{\alpha_j}$ for any $j \in \{k+1,\ldots,\ell\}$.

\begin{lemma}  \label{l:eq}

{\rm (i)} For all $i \in \{ 1,\ldots, k\}$, one has $m_i ' \le m_i$
and, for all $j \in \{ k+1,\ldots, \ell \}$, one has $a_{j} \le m_j +1$.
In particular,
$$
	| W_I | \, a_{k+1} \cdots  a_\ell
		\le  | W_{S}|.
$$

{\rm (ii)} Equality
holds in the above inequality if and only if $I$ and $J$
are in one of the configurations {\rm (C1)}, {\rm (C2)}
or {\rm (C3)} as described in Proposition \ref{p:PP}
with $\F=J$.

\end{lemma}

\begin{proof}

(i) By Lemma \ref{l:ht},
for all $i \in \{ 1,\ldots, k\}$, we have $m_i ' \le m_i$.
Turn to the second statement.
Set for $j \in \{k+1, \ldots, \ell \}$,
$I_{j} := I \cup \{ \alpha_{k+1} , \ldots, \alpha_j\}$.
Let $j \in \{k+1, \ldots, \ell \}$
and $S_j$ the connected component of $I_j$ containing $\alpha_j$.
We have
$a_j = 2 -\langle \rho_{I} ,\check{\alpha} \rangle
= 2 -\langle \rho_{I \cap S_j} ,\check{\alpha} \rangle
= 2\langle \rho_{S_j} - \rho_{I \cap S_j} ,\check{\alpha} \rangle$.
So, by Lemma \ref{l:root},
$a_j \le {\rm ht}(\theta_{S_j} ) + 1$.
Hence, by Lemma \ref{l:ht}, $a_j \le m_j +1$ since $I_j$ has cardinality $j$.
All this shows:
$$Ê| W_I |   \, \prod_{j=k+1}^{\ell} a_{j}  = \prod_{i=1}^{k}( m_i '  +1)  \prod_{j=k+1}^{\ell} a_{j}
\le \prod_{i=1}^{\ell}( m_i  +1) = |W_S |.
$$

\medskip

(ii) By the proof of (i),
if equality holds in the above inequality then $|W_I | = \prod_{i=1}^{k}( m_i   +1) $
and for all $j \in \{ k+1,\ldots, \ell \}$, $a_{j} = m_j +1.$
In particular, $a_\ell = m_\ell +1$.
Therefore, we are in one of the situations
of the Lemma \ref{l:root}
and we consider the six cases as described in it.

\textbullet \; Type ${\bf A}_\ell$, $\ell \ge 1$: The $\ell-1$ smallest degrees of the basic invariants are
$2,3\ldots, \ell$.
If $\alpha_\ell$ is not an extremal vertex,
then $| W_{S \smallsetminus \{\alpha_\ell\}} | < \ell \, !$ as we easily verify.
So $\alpha_\ell$ must be extremal and $I$ and $J$ are in the configuration (C1).

\textbullet \; Type ${\bf B}_\ell$,  $\ell \ge 2$:
The $\ell-1$ smallest degrees of the basic invariants are
$2,4,\ldots, 2 (\ell-1)$.
So their product is strictly greater than $|W_{S \smallsetminus \{\beta_\ell\}} | = \ell ! $
and the equality does not hold.

\textbullet \; Type ${\bf C}_\ell$, $\ell \ge 3$: $I$ and $J$ are in the configuration (C2).

\textbullet \; Type ${\bf D}_\ell$, $\ell \ge 4$:
The degrees of the basic invariants of ${\bf D}_{\ell}$, for $\ell \ge 4$, are
$2,4,\ldots, 2 \ell -2, \ell$.
So, the $\ell-1$ smallest are $2,4,\ldots, 2 \ell - 4, \ell$ ($\ell \ge 4$)
and their product is $2^{\ell - 2} \ell$.
But for any $i \in \{1,\ldots,\ell\}$,
$| W_{S \smallsetminus \{\beta_i\}} | \le | W_{S \smallsetminus \{\beta_1\}} | =2^{\ell -2} (\ell-1) ! <  2^{\ell - 2} \ell$;
so the equality does not hold.

\textbullet \; Type ${\bf E}_6$:
The 5-th smallest exponents of ${\bf E}_{6}$ are
$1,4,5,7,8$ and those of $S \smallsetminus \{ \beta_1\}$
(or of $S \smallsetminus \{ \beta_6\}$)
are $1,3,4,5,7$; so, the equality does not hold.

\textbullet \; Type ${\bf E}_7$:
The 6-th smallest exponents of ${\bf E}_{7}$ are
$1,5,7,9,11,13$ and those of $S \smallsetminus \{ \beta_7 \}$
are $1,4,5,7,8,11$; so, the equality does not hold.

\smallskip

One has proven one implication.
The converse implication is an easy computation,
left to the reader.

\end{proof}

\begin{prop}  \label{p:Eul}

Assume that $X$ is a simple locally factorial $G/H$-embedding with maximal colored cone $(\sigma,\F)$
of dimension $r$.
Let $I$ be the subset of $S$ such that $N_G(H) = P_I$.
Then,

$$ e_{\rm st} (X) = \displaystyle{\frac{ | W_{S} | }{  | W_{I} | \,
		\prod_{ \alpha \in \F} Êa_\alpha }} \quad
		\textrm{ and } \quad
	e(X) = \displaystyle{\frac{ | W_{S} | }{ | W_{I \cup \F} |}}  \, .
$$

\end{prop}

\begin{proof}

First of all, observe that the Euler number of $G/B$
is the number of fixed points of a maximal torus, i.e.,
the order of the Weyl group $W_S$.
More generally, for any $S' \subset S$, the Euler number of
$G/P_{S'}$ is $|W_S| / |W_{S'}|$.
Thus, we have to show:
$$ e_{\rm st} (X) = \displaystyle{\frac{ e(G/P_I )}{
		\prod_{ \alpha \in \F} Êa_\alpha }}   \quad
\textrm{ and } \quad
e(X) = e(G/ P_{I \cup \F} )
\, .
$$
Now, we observe that the usual Euler number of a horospherical homogeneous space
is nonzero if and only if it has rank zero.
As a consequence, one has $e(X) = e(G /P_{I \cup \F})$,
according to the description of $G$-orbits in $X$ (see Proposition \ref{p:orb}).

Turn to the formula for $e_{\rm st}(X)$.
Let $e_1,\ldots,e_r$ be the primitive generators of $\sigma$.
Since $X$ is locally factorial, $e_1,\ldots,e_r$ is
a $\Z$-basis of $\sigma \cap N$ (cf. Theorem \ref{t:lf}).
Then
$$
	\sum\limits_{e_i \in \sigma \cap N} \L^{\omega_X (e_i)}
	 =  \prod_{i=1}^r \displaystyle{\frac{1}{1 - \L^{\omega_X (e_i)}}}
	= \displaystyle{\frac{1}{ (\L -1)^r}}
		\prod_{i=1}^r \displaystyle{\frac{\L^{ - \omega_X(e_i)} }{\L^{-\omega_X (e_i) -1 } + \cdots +1}} .
$$
Then, by Theorem\,\ref{t:main}, one has:
\begin{eqnarray*}
	\E_{\rm st}(X ) \ = \ [G/H] \sum\limits_{e_i \in \sigma \cap N} \L^{\omega_X (e_i)}
		& = & [ G/P] \, [T] \ \displaystyle{\frac{1}{ (\L -1)^r}}
		\prod_{i=1}^r \displaystyle{\frac{\L^{ - \omega_X(e_i)} }{\L^{-\omega_X (e_i) -1 } + \cdots +1}} \\
		& = &  [ G/P ]
		\prod_{i=1}^r \displaystyle{\frac{\L^{ - \omega_X(e_i)} }{\L^{-\omega_X (e_i) -1 } + \cdots +1}} \, .
\end{eqnarray*}
From this, we get
$$
	e_{\rm st} (X) = e(G/P) \prod_{i=1}^r \displaystyle{\frac{1}{\big( -\omega_X(e_i) \big)}}
	= \displaystyle{\frac{e(G/P_I)}{\prod_{\alpha \in \F} a_\alpha}} \, .
$$
The last equality holds because the set of elements $\varrho_\alpha$
($\alpha \in \F$)
is a subset of the basis $\{e_1,\ldots,e_r\}$ (cf. Theorem \ref{t:lf}).

\end{proof}

We are in a position to prove Theorem\,\ref{t:smo}.

\begin{proof}[Proof of Theorem\,\ref{t:smo}]

We can assume without loss of generality that $S$ is connected
and $I \cup \F = S$.
By Lemma \ref{l:eq} and Proposition \ref{p:Eul},
we have $e_{\rm st}(X) \ge e(X) $.
This proves one part of the theorem.
Moreover, the equality holds if and only if $(I,\F)$ is
in one of the configurations (C1), (C2)
or (C3) as described in Proposition \ref{p:PP},
that is to say if and only if $X$ is smooth by Proposition \ref{p:PP}.

\end{proof}

\begin{rem}

As a matter of fact,
we gave another proof for the first implication
of Pauer's criterion (Proposition \ref{p:PP}).
Indeed, whenever $(I,\F)$ is not in one of the configurations (C1), (C2)  or (C3)
of Proposition \ref{p:PP}, we have shown that $e_{\rm st}(X) > e(X)$,
and so $X$ is not smooth.

\end{rem}

\section{Some applications and open questions}                   \label{S:App}

Let $X$ be a complete locally factorial horospherical $G/H$-embedding
with colored fan $\Sigma$.
Let $e_1,\ldots, e_s$ be the primitive integral
generators of all $1$-dimensional 
cones in $\Sigma$
and set $a_i : = -\omega_X (e_i)$ for all $i \in \{Ê1,\ldots, s\}$.

Consider the polynomial
ring $\C[z_1,\ldots,z_s]$ whose
variables $z_1,\ldots,z_s$ are in
bijection with the lattice vectors $e_1,\ldots,e_s$.
Recall that the Stanley-Reisner ring $R_\Sigma$ is
the quotient of $\C[z_1,\ldots,z_s]$
by the ideal generated by all square free monomials
$z_{i_1} \ldots z_{i_k}$
such that the lattice vectors
$e_{i_1} \ldots e_{i_k}$ do not generate any
$k$-dimensional cone in $\Sigma$.
Recall also that the \emph{weighted Stanley-Reisner ring}
$R_\Sigma^w$ is  defined
by putting $\deg z_i = a_i$
in the standard Stanley-Reisner ring $R_\Sigma$.

\begin{prop}       \label{p:SR}

Let $X$ be a complete locally factorial horospherical $G/H$-embedding
with colored fan $\Sigma$.
Then, one has:
\begin{eqnarray}   \label{eq:P}
	&& \sum\limits_{n\in N} (uv)^{\omega_X(n)}
	= P(R_\Sigma^w , (uv)^{-1})
	=  \sum\limits_{\sigma \in \Sigma}
		\displaystyle{ \frac{(-1)^{\dim \sigma}}{ \prod_{e_i \in \sigma} \big(1 - (uv)^{a_i} \big)} }  \, ;
			\\ \label{eq:ai}
	&& E_{\rm st} (X ; u,v ) = E(G/H ; u,v) (-1)^r P(R_\Sigma^w ,uv) \, ,
\end{eqnarray}
where $P(R_\Sigma^w, t)$
denotes the Poincar\'e series of the weighted Stanley-Reisner ring $R_\Sigma^w$.

\end{prop}

\begin{proof}

The ring $R_\Sigma$ has a monomial basis over $\C$
whose elements are in one-to-one correspondence
with $N$.
Namely,
any monomial $z_{i_1}^{k_1} \ldots z_{i_t}^{k_t}$
in $R_\Sigma$ corresponds to the
lattice point
$k_{1} e_{i_1} + \cdots + k_t e_{i_t}$
and the weighted degree of  $z_{i_1}^{k_1} \ldots z_{i_t}^{k_t}$
is $- k_{1} \omega_X(e_{i_1}) - \cdots  - k_t \omega_X(e_{i_t})$.
Thus, the $k$-homogeneous component of the weighted Stanley-Reisner
ring $R_{\Sigma}^{w}$ consists of all
monomials $z_{i_1}^{k_1} \ldots z_{i_t}^{k_t}$
corresponding
to lattice points $n \in N$ such that
$\omega_X(n) = -k$.
This implies the first equality in (\ref{eq:P}).
For any cone $\sigma \in \Sigma$, we denote by
$\sigma^\circ$ the relative interior of $\sigma$.
Since $X$ is locally factorial, one has by Theorem \ref{t:lf}:
\begin{eqnarray}     \label{eq:P2}
	\sum\limits_{n \in N} t^{\omega_X (n)}
	& =&  \sum\limits_{\sigma \in \Sigma} \sum_{n \in \sigma^\circ} t^{\omega_X (n)}
	= \sum\limits_{\sigma\in \Sigma} \prod_{e_i \in \sigma} \displaystyle{ \frac{t^{-a_i}}{1 - t^{-a_i}}}
	= \sum\limits_{\sigma\in \Sigma} \prod_{e_i \in \sigma} \displaystyle{ \frac{(-1)^{\dim \sigma}}{1 - t^{a_i}}} \, .
\end{eqnarray}
This implies the second equality in (\ref{eq:P}).

\smallskip

Let us prove the equality (\ref{eq:ai}).
By Theorem \ref{t:main} and (\ref{eq:P}), we have:
$$
	 E_{\rm st} (X ; u ,v ) =  E(G/H ; u,v) P(R_\Sigma^w, (u v)^{-1}) .
$$
By the Poincar\'e duality \cite[Theorem 3.7]{Ba98},
we have
\begin{eqnarray*}
	&& (uv)^{\dim X} E_{\rm st} (X ; u^{-1} ,v^{-1} )
		 = E_{\rm st} (X ;u,v), \\
	&& (uv)^{\dim G/P} E (G/P ; u^{-1} ,v^{-1} )
		= E (G/P ;u,v) \, .
\end{eqnarray*}
The above equalities imply:
\begin{eqnarray*}
	 E_{\rm st} (X ; u ,v )  & = & (uv)^{\dim X} E_{\rm st} (X ; u^{-1} ,v^{-1} ) \\
	 & = &  (uv)^{\dim X} E(G/H ; u^{-1},v^{-1})  P(R_\Sigma^w, u v) \\
	& = &(uv)^{\dim G/P} E(G/P ; u^{-1},v^{-1}) (uv)^{r} ((uv)^{-1} -1)^r P(R_\Sigma^w, u v) \\
	 & = &  E(G/P ; u ,v) (uv-1)^r (-1)^r P(R_\Sigma^w, u v) \\
	& = & E(G/H ; u,v) (-1)^r P(R_\Sigma^w, u v)\, .
\end{eqnarray*}

\end{proof}

\begin{ex}

1) Consider the locally factorial
completion $\overline{Q}$ of the affine 5-dimensional quadric $Q$ in Example \ref{ex:Q};
$\overline{Q}$ is a singular projective quadric.
The colored fan $\overline{\Sigma}$ of $\overline{Q}$ is represented in Figure \ref{fig:1}
and  the positive integer $a_i = -\omega_{\overline{Q}}(e_i)$ ($i=1,2,3$)
is written down near to  the integral point $e_i$.
The circles stand for the colors $\varrho_\alpha$,
$\alpha \in \F$.

\begin{figure}[htb]
{\setlength{\unitlength}{0.1in}

\begin{center}
\begin{picture}(8,8)(0,0)

\put(4,4){\line(1,0){4}}
\put(4,4){\line(0,1){4}}
\put(4,4){\line(-1,-1){4}}

\put(6,4){\circle{0.8}}
\put(4,6){\circle{0.8}}

\put(6,4){\circle*{0.4}}
\put(4,6){\circle*{0.4}}

\put(2,2){\circle*{0.4}}

\put(1,1.9){\tiny$1$}
\put(6.1,2.8){\tiny$2$}
\put(2.8,5.9){\tiny$2$}

\qbezier[15](0,2)(4,2)(8,2)
\qbezier[15](0,4)(4,4)(8,4)
\qbezier[15](0,6)(4,6)(8,6)

\qbezier[15](2,0)(2,4)(2,8)
\qbezier[15](4,0)(4,4)(4,8)
\qbezier[15](6,0)(6,4)(6,8)

\qbezier[10](5,4)(4.5,4.5)(4,5)
\qbezier[20](6,4)(5,5)(4,6)
\qbezier[30](7,4)(5.5,5.5)(4,7)
\qbezier[40](8,4)(6,6)(4,8)
\qbezier[30](8,5)(6.5,6.5)(5,8)
\qbezier[20](8,6)(7,7)(6,8)
\qbezier[10](8,7)(7.5,7.5)(7,8)

\qbezier[40](4,4)(4,2)(4,0)
\qbezier[40](4.5,4)(4.5,2)(4.5,0)
\qbezier[40](5,4)(5,2)(5,0)
\qbezier[40](5.5,4)(5.5,2)(5.5,0)
\qbezier[40](6,4)(6,2)(6,0)
\qbezier[40](6.5,4)(6.5,2)(6.5,0)
\qbezier[40](7,4)(7,2)(7,0)
\qbezier[40](7.5,4)(7.5,2)(7.5,0)
\qbezier[40](8,4)(8,2)(8,0)

\qbezier[35](3.5,3.5)(3.5,1.75)(3.5,0)
\qbezier[30](3,3)(3,1.5)(3,0)
\qbezier[25](2.5,2.5)(2.5,1.25)(2.5,0)
\qbezier[15](2,2)(2,1)(2,0)
\qbezier[10](1.5,1.5)(1.5,.75)(1.5,0)
\qbezier[5](1,1)(1,0.5)(1,0)
\qbezier[3](0.5,0.5)(0.5,0.25)(0.5,0)

\qbezier[40](4,4)(2,4)(0,4)
\qbezier[40](4,4.5)(2,4.5)(0,4.5)
\qbezier[40](4,5)(2,5)(0,5)
\qbezier[40](4,5.5)(2,5.5)(0,5.5)
\qbezier[40](4,6)(2,6)(0,6)
\qbezier[40](4,6.5)(2,6.5)(0,6.5)
\qbezier[40](4,7)(2,7)(0,7)
\qbezier[40](4,7.5)(2,7.5)(0,7.5)
\qbezier[40](4,8)(2,8)(0,8)

\qbezier[35](3.5,3.5)(1.75,3.5)(0,3.5)
\qbezier[30](3,3)(1.5,3)(0,3)
\qbezier[25](2.5,2.5)(1.25,2.5)(0,2.5)
\qbezier[15](2,2)(1,2)(0,2)
\qbezier[10](1.5,1.5)(.75,1.5)(0,1.5)
\qbezier[5](1,1)(0.5,1)(0,1)
\qbezier[3](0.5,0.5)(0.25,0.5)(0,0.5)

\end{picture}
\end{center}}

\caption{The colored fan $\overline{\Sigma}$ of $\overline{Q}$} \label{fig:1}
\end{figure}
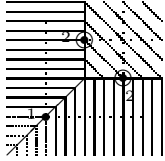

\noindent
The Stanley-Reisner ring is $R_{\overline{\Sigma}} \simeq \C[z_1,z_2,z_3]/ {(z_1 z_2 z_3)}$
and we have
$$
	P (R_{\overline{\Sigma}}^w , t) =  \displaystyle{\frac{1 - t^5 }{ (1-t) (1-t^2)^2}}.
$$
Hence, by Proposition \ref{p:SR}, we get
$$
	E_{\rm st}(\overline{Q} ; u,v )
	= \displaystyle{\frac{(1 + uv + (uv)^2)(1+ uv +(uv)^2+(uv)^3) }{ (1+ uv)}} \, .
$$

\smallskip

2) Consider the locally factorial
completion $\overline{X}$ of the affine 7-dimensional cone $X$ over the Grassmannian
$G(2,5)$ from Example \ref{ex:Grass}; $\overline{X}$ is the projective cone over
the Grassmannian $G(2,5)$.
The colored fan $\overline{\Sigma}$ of $\overline{X}$ is represented in Figure \ref{fig:2}.

\begin{figure}[htb]
{\setlength{\unitlength}{0.1in}

\begin{center}
\begin{picture}(8,8)(0,0)

\put(4,4){\line(1,0){4}}
\put(4,4){\line(0,1){4}}
\put(4,4){\line(-1,-1){4}}

\put(6,4){\circle{0.8}}
\put(4,6){\circle{0.8}}

\put(6,4){\circle*{0.4}}
\put(4,6){\circle*{0.4}}

\put(2,2){\circle*{0.4}}

\put(1,1.9){\tiny$1$}
\put(6.1,2.8){\tiny$2$}
\put(2.8,5.9){\tiny$3$}

\qbezier[15](0,2)(4,2)(8,2)
\qbezier[15](0,4)(4,4)(8,4)
\qbezier[15](0,6)(4,6)(8,6)

\qbezier[15](2,0)(2,4)(2,8)
\qbezier[15](4,0)(4,4)(4,8)
\qbezier[15](6,0)(6,4)(6,8)

\qbezier[10](5,4)(4.5,4.5)(4,5)
\qbezier[20](6,4)(5,5)(4,6)
\qbezier[30](7,4)(5.5,5.5)(4,7)
\qbezier[40](8,4)(6,6)(4,8)
\qbezier[30](8,5)(6.5,6.5)(5,8)
\qbezier[20](8,6)(7,7)(6,8)
\qbezier[10](8,7)(7.5,7.5)(7,8)

\qbezier[40](4,4)(4,2)(4,0)
\qbezier[40](4.5,4)(4.5,2)(4.5,0)
\qbezier[40](5,4)(5,2)(5,0)
\qbezier[40](5.5,4)(5.5,2)(5.5,0)
\qbezier[40](6,4)(6,2)(6,0)
\qbezier[40](6.5,4)(6.5,2)(6.5,0)
\qbezier[40](7,4)(7,2)(7,0)
\qbezier[40](7.5,4)(7.5,2)(7.5,0)
\qbezier[40](8,4)(8,2)(8,0)

\qbezier[35](3.5,3.5)(3.5,1.75)(3.5,0)
\qbezier[30](3,3)(3,1.5)(3,0)
\qbezier[25](2.5,2.5)(2.5,1.25)(2.5,0)
\qbezier[15](2,2)(2,1)(2,0)
\qbezier[10](1.5,1.5)(1.5,.75)(1.5,0)
\qbezier[5](1,1)(1,0.5)(1,0)
\qbezier[3](0.5,0.5)(0.5,0.25)(0.5,0)

\qbezier[40](4,4)(2,4)(0,4)
\qbezier[40](4,4.5)(2,4.5)(0,4.5)
\qbezier[40](4,5)(2,5)(0,5)
\qbezier[40](4,5.5)(2,5.5)(0,5.5)
\qbezier[40](4,6)(2,6)(0,6)
\qbezier[40](4,6.5)(2,6.5)(0,6.5)
\qbezier[40](4,7)(2,7)(0,7)
\qbezier[40](4,7.5)(2,7.5)(0,7.5)
\qbezier[40](4,8)(2,8)(0,8)

\qbezier[35](3.5,3.5)(1.75,3.5)(0,3.5)
\qbezier[30](3,3)(1.5,3)(0,3)
\qbezier[25](2.5,2.5)(1.25,2.5)(0,2.5)
\qbezier[15](2,2)(1,2)(0,2)
\qbezier[10](1.5,1.5)(.75,1.5)(0,1.5)
\qbezier[5](1,1)(0.5,1)(0,1)
\qbezier[3](0.5,0.5)(0.25,0.5)(0,0.5)

\end{picture}
\end{center}}

\caption{The colored fan $\overline{\Sigma}$ of $\overline{X}$} \label{fig:2}
\end{figure}
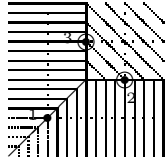

\noindent
We have,
$$
	P (R_{\overline{\Sigma}}^w , t) =  \displaystyle{\frac{1 - t^6 }{ (1-t)(1-t^2)(1-t^3)}} \, ,
$$
and
$$
	E_{\rm st}(\overline{X} ; u,v ) = (1 + (uv)^2)(1 + uv + (uv)^2 + (uv)^3 + (uv)^4 + (uv)^5) .
$$

\end{ex}

\smallskip

It would be interesting to compute the cohomology ring $H^\ast (X_\Sigma , \C)$
of an arbitrary smooth projective horospherical variety $X_\Sigma$
defined by a colored fan $\Sigma$.
If $X_\Sigma$ is a toroidal horospherical variety, then $X_\Sigma$ is a toric bundle
over $G/P$, and  a general result of Sankaran
and Uma \cite[Theorem 1.2]{SU03} implies the following description
of the cohomology ring of $X_\Sigma$:

\begin{prop}    \label{p:coh}

Let $X_\Sigma$ be a smooth projective toroidal horospherical variety
defined by a (uncolored) fan $\Sigma$.
Then the cohomology ring $H^\ast (X_\Sigma , \C)$ is isomorphic
to the quotient of  $H^\ast (G/P ,\C) \otimes_\C R_\Sigma$
by the ideal generated by the regular sequences $f_1,\ldots,f_r$
where $f_i $ is given by
$$
	f_i :=  \delta(m_i) \otimes 1 + 1\otimes  \sum_{j=1}^s \langle m_i , e_j \rangle
		\in \big( H^2 (X, \C) \otimes R_\Sigma^0 \big) \oplus
		\big( H^0 (X, \C) \otimes R_\Sigma^1 \big) \, ,
$$
for some integral basis $\{ m_1, \ldots, m_r \}$ of the lattice $M$.

\end{prop}

Together with Proposition \ref{p:coh}, our formula (\ref{eq:ai}) in Proposition \ref{p:SR} motivates the following
question:

\begin{question}        \label{q:coh}

Does there exist an analogous description of the cohomology
ring of an arbitrary smooth projective horospherical variety
defined by a colored fan $\Sigma$
which involves the weighted Stanley-Reisner ring $R_\Sigma^w$?

\end{question}

Another interesting question is motivated by Theorem \ref{t:main}:

\begin{question}     \label{q:st}

How to compute  $E_{\rm st}(X ; u,v)$
for an arbitrary $\Q$-Gorenstein spherical $G/H$-embedding?

\end{question}

\begin{rem}
We hope that there is  a formula for $E_{\rm st}(X ; u,v)$ similar to the one
in the horospherical case, e.g.,  which involves
the summation of $(uv)^{\omega_X(n)}$
over all lattice points
in the valuation cone $\mathcal{V}(G/H)$
of the spherical homogeneous space $G/H$.
\end{rem}

A smoothness criterion for arbitrary spherical varieties
was obtained by M. Brion in \cite{Br91}.
Unfortunately, this criterion is difficult to apply in practice.
We expect that the smoothness criterion
for locally factorial horospherical varieties (see Theorem \ref{t:smo})
can be extended to arbitrary locally factorial spherical varieties:

\begin{conj}       \label{c:smo}

Let $X$ be a locally factorial spherical $G/H$-embedding
whose closed orbits are projective.
Then one has $e_{\rm st}(X) \ge e (X)$, and
the equality holds if and only if $X$ is smooth.

\end{conj}

\end{document}